\documentclass[12pt,twoside]{amsart}
\usepackage{amsmath,amssymb,mathrsfs}

\usepackage{color}

\newtheorem{theorem}{Theorem}[section]
\newtheorem{lemma}[theorem]{Lemma}
\newtheorem{proposition}[theorem]{Proposition}
\newtheorem{corollary}[theorem]{Corollary}
\newtheorem{definition}[theorem]{Definition}

\newtheorem{rmrk}[theorem]{Remark}

\DeclareMathAlphabet{\mathbfit}{OML}{cmm}{b}{it}

\makeatletter
\@addtoreset{equation}{section}
\makeatother

\newenvironment{remark}
{\begin{rmrk} \em}
{\end{rmrk}}

\newcommand{\fn} {function}

\newcommand{\tr} {trajector}

\newcommand{\sy} {system}

\newcommand{\R} {\mathbb{R}}
\newcommand{\C} {\mathbb{C}}

\newcommand{\Z} {\mathbb{Z}}
\newcommand{\N} {\mathbb{N}}

\newcommand{\article}[3] {\textsc{{#1}}, {\itshape {#2}}, {{#3}}.}
\newcommand{\book}[3] {\textsc{{#1}}, {\itshape {#2}}, {{#3}}.}
\newcommand{\vol} {\textbf}
\newcommand{\eps} {\varepsilon}

\newcommand{\into} {\longrightarrow}

\renewcommand{\emptyset} {\varnothing}

\newcommand{\pr} {probability}
\newcommand{\ra} {random}
\newcommand{\rw} {random walk}
\newcommand{\en} {environment}
\newcommand{\sgn} {\mathrm{sgn}}

\newcommand{\Oen} {\Omega_\mathrm{en}}   

 \newcommand{\be}{\begin{equation}}
 \newcommand{\ee}{\end{equation}}

 \newcommand{\ba}{\begin{array}}
 \newcommand{\ea}{\end{array}}

 \newcommand{\bea}{\begin{eqnarray}}
 \newcommand{\eea}{\end{eqnarray}}

 \newcommand{\bl}{\begin{lemma}}
 \newcommand{\el}{\end{lemma}}

 \newcommand{\br}{\begin{remark}}
 \newcommand{\er}{\end{remark}}

 \newcommand{\bt}{\begin{theorem}}
 \newcommand{\et}{\end{theorem}}

 \newcommand{\bd}{\begin{definition}}
 \newcommand{\ed}{\end{definition}}

 \newcommand{\bcl}{\begin{claim}}
 \newcommand{\ecl}{\end{claim}}

 \newcommand{\bp}{\begin{proposition}}
 \newcommand{\ep}{\end{proposition}}

 \newcommand{\bc}{\begin{corollary}}
 \newcommand{\ec}{\end{corollary}}

 \newcommand{\bpr}{\begin{proof}}
 \newcommand{\epr}{\end{proof}}

 \newcommand{\bi}{\begin{itemize}}
 \newcommand{\ei}{\end{itemize}}

 \newcommand{\ben}{\begin{enumerate}}
 \newcommand{\een}{\end{enumerate}}

 \def \D {{\mathbb D}}

 \def \P {{\mathbb P}}
 \def \E {{\mathbb E}}

 \def \cN {\mathcal{N}}

 \def \cT {\mathcal{T}}
 
 \def \cV {\mathcal{V}}

 \def \a {{\alpha}}
 \def \b {{\beta}}

 \def \e {{\varepsilon}}

  \def \g {{\gamma}}
 
 \def \s {{\sigma}}
 
 \def \z {{\z}}

 \def \t {{\tau}}

 \def \o {{\omega}}

 \def \x {{\xi}}
 \def \z {{\zeta}}

\def \à {{\`{a}}}
\def \ì{{\`{\i}}}
\def \ò{{\`{o}}}
\def \è{{\`{e}}}
\def \ù{{\`{u}}}

 \def \1{\mathbbm{1}} 

\begin{document}

\title[Random walk between L\'evy-spaced targets]{Continuous-time
random walk between L\'evy-spaced targets in the real line}

\author{Alessandra Bianchi}
\address{Dipartimento di Matematica,
Universit\`a di Padova, Via Trieste 63,
35121 Padova, Italy.}\email{bianchi@math.unipd.it}
\author{Marco Lenci}
\address{Dipartimento di Matematica, Universit\`a di Bologna,
Piazza di Porta San Donato 5, 40126 Bologna, Italy\\
and Istituto Nazionale di Fisica Nucleare,
Sezione di Bologna, Via Irnerio 46,
40126 Bologna, Italy.}
\email{marco.lenci@unibo.it}
\author{Fran\c coise P\`ene}
\address{Institut Universitaire de France and Universit\'e de Brest,
UMR CNRS 6205, Laboratoire de Math\'ematique de Bretagne
Atlantique, 6 avenue Le Gorgeu, 29238 Brest cedex, France.}
\email{francoise.pene@univ-brest.fr}.

\date{Mar 26, 2019}

\begin{abstract}
  We consider a continuous-time random walk which is defined as
  an interpolation of a random walk on a point process on the real
  line. The distances between neighboring points of the point process
  are i.i.d.\ random variables in the normal domain of attraction of an
  $\alpha$-stable distribution with $0 < \alpha < 1$. This is therefore
  an example of a \emph{random walk in a L\'evy random medium}.
  Specifically, it is a generalization of a process known in the physical
  literature as \emph{L\'evy-Lorentz gas}. We prove that the annealed
  version of the process is superdiffusive with scaling exponent
  $1/(\a+1)$ and identify the limiting process, which is not c\`adl\`ag.
  The proofs are based on the technique of Kesten and Spitzer for
  random walks in random scenery.

  
\end{abstract}

\maketitle

\section{Introduction}
\label{intro}

In this paper we study a continuous-time \rw\ in a \ra\ medium on the
real line. The \ra\ medium is given by a marked point process in which
the distances between consecutive points are i.i.d.\ variables taken in
the basin of attraction of an $\a$-stable distribution, with $0 < \a < 1$.
The marked points are referred to as \emph{targets}. A particle travels
in $\R$ with unit speed visiting the sequence of targets selected by an
independent \rw\ $(S_n, \, n\in \N)$. This \rw\ is called the
\emph{underlying \rw}.

So, for example, if the realization of $(S_n)$ is $(0, 2, -3, -1, \ldots)$,
the particle will travel with velocity $1$ from the origin $O$ to the
second target to the right of $O$. Then it will instantaneously change
its velocity to $-1$ to travel toward the third target to the left of $O$.
Then it will travel with velocity $1$ toward the first target to the left
of $O$, and so on.

Since the distances between targets are fat-tailed variables (their
first moment is infinite), the \tr ies of the process occasionally
experience extremely long inertial segments, leading one to believe
that the process is superdiffusive. By that we mean that it scales like
a power of time with exponent bigger than $1/2$. The purpose of the
paper is to show this in a very specific sense.

More precisely, if $X(t) \equiv X^\o(t)$ denotes our process, with the
label $\o$ representing the \ra\ medium, we prove that the
finite-dimensional distributions of $(n^{-1/(\a+1)} X(nt), \, t \ge 0)$,
w.r.t.\ both the \ra\ medium and the \ra\ dynamics, converges to those of a certain
process which we identify (Theorem \ref{thm1}). In other words, we
prove an annealed generalized CLT for the finite-dimensional
distributions of $(X(t), \, t \ge 0)$.

Processes of this kind have wide application in the physical sciences,
where they are used as models for \emph{anomalous diffusion}; see
\cite{szu, krs, zdk, cgls, vbb} and references therein.
Indeed, in many real-world applications (in statistical physics, optics,
epidemics, etc.), the long inertial segments, or \emph{ballistic flights},
that cause superdiffusion are not due to an anomalous mechanism
that governs the dynamics of the agent (a particle, a photon, an animal,
etc.)\ but rather to the complexity of the medium in which the motion
occurs. An example is molecular diffusion in porous media \cite{le} (see
the reference lists of \cite{le, bfk, bcv} for more examples).
As a matter of fact, the model presented here is a generalization
of the so-called \emph{L\'evy-Lorentz gas}, which was introduced
in \cite{bfk} precisely as a one-dimensional toy model for transport
in porous media. (The L\'evy-Lorentz gas is the case
where the underlying \rw\ is a simple symmetric \rw.)

In the language of \pr\ it makes sense to call the process at hand
a continuous-time \emph{\rw\ in a L\'evy \ra\ medium}, meaning that
it is the medium that causes the long ballistic flights. We occasionally
also use the expression `\rw\ in a L\'evy \ra\ \en' to lay emphasis on the
analogies between our process and the \rw s in \ra\ \en\ ---
even though the two types of processes are technically different.

A rigorous study of our \sy\ was initiated in \cite{bcll} where the authors
considered the easier case where the distances between the targets
have finite mean and infinite variance. This includes all random variables
with distribution in the basin of attraction of an $\a$-stable law with
$1 < \a < 2$. For this regime, subject to certain hypotheses, they
prove a quenched normal CLT, that is, for a.e.\ realization $\o$ of the
medium, $t^{-1/2} X^\o(t)$ converges, as $t \to \infty$, to a Gaussian
variable independent of $\o$. This implies in particular the annealed
normal CLT,  confirming and adding on some of the predictions of
\cite{bcv} for the L\'evy-Lorentz gas.

Our main result in this paper confirms and extends some other
predictions of \cite{bcv} and for the first time, to our knowledge,
prove superdiffusion for some processes in L\'evy \ra\ media.
The scaling we obtain is also in agreement with the one that has
been identified recently for a related model \cite{acor}. Our
proofs are based on an adaptation of the technique of Kesten and
Spitzer for \emph{random walks in random scenery} \cite{ks}.
This explains why our limit processes involve the so-called
\emph{Kesten-Spitzer process} $\Delta$.

The paper is organized as follows: In Section \ref{sec-setup} we give
the precise definitions of all the processes associated with our random
walk, and present our most important results (Theorems \ref{thm1} and
\ref{thm2}). In Section \ref{sec-flow} we prove these results by means
of what we call our Main Lemma (Lemma \ref{main-lemma}). The Main
Lemma is then proved in Section \ref{sec-main}, using ideas from the
theory of random walks in random scenery.

\bigskip\noindent
\textbf{Acknowledgments.}\ F.~P\`ene is supported by the IUF and by
the ANR Project MALIN (ANR-06-TCOM-0003). M.~Lenci's research is
part of his activity within the Gruppo Nazionale di Fisica Matematica (INdAM).
We also wish to thank the American Institute of Mathematics for the
Workshop \emph{``Stochastic Methods for Non-Equilibrium Dynamical
Systems"}, where this work was initiated.

\section{Setup}
\label{sec-setup}

Let $\z := (\z_j, \, j \in \Z )$ be a sequence of i.i.d.\ positive \ra\ variables,
the common distribution of which belongs to the normal basin of
attraction of an $\a$-stable distribution, with $0 < \a < 1$. This means
that, as $n \to \infty$, $n^{-1/\a} \sum_{j=1}^n \z_j$ converges in
distribution to a variable $Z_1$, which must then be positive and
non-integrable.  The recursive sequence of definitions
\be \label{PP}
  \o_0 := 0 \, , \qquad \o_k - \o_{k-1} := \z_k \,, \quad \mbox{for } k
  \in \Z \,.
\ee
determines a marked \emph{point process} $\o := ( \o_k, \, k\in\Z)$
on $\R$, which we call the \emph{\ra\ medium} or \emph{\ra\ \en}, or
simply the \en. Each point $\o_k$ will be called a \emph{target}. The
set of all possible \en s is denoted $\Oen$, and the law on it $P$.

Let $S := (S_n,\, n \in \N)$ be a $\Z$-valued random walk starting
from $S_0 := 0$ with centered i.i.d.\ increments
$\xi_j := S_j - S_{j-1}$, for $j \in \Z^+$. We assume that there exists
$\g > 2/\a$ such that the absolute moment of $\xi_j$ of order $\g$ is
finite
(this assumption is used in Lemma \ref{lemma-Nn}).
This implies that the variance of $\xi_j$, which we denote
$v_\xi$, is also finite. In order for the problem to make sense, we also
assume $v_\xi>0$. We refer to $S$ as the \emph{underlying
random walk} and call $Q$ its distribution (on $\Z^\N$
endowed with the $\sigma$-algebra generated by cylinders).

Given $\o \in \Oen$, we define the so-called \emph{\rw\ on the point
process} to be, for all $n \in \N$,
\be\label{Y}
  Y_n \equiv Y_n^\o := \o_{S_n}\,.
\ee
In simple terms, $Y := (Y_n,\, n \in \N)$ performs the same jumps
as $S$, but on the points of $\o$.

The process we are interested in here is the continuous-time \rw\
$X \equiv X^\o := ( X^\o (t), \, t \ge 0 )$ whose trajectories interpolate
those of the walk $Y$ and have unit speed (save at collision times).
Formally it can be defined as follows: Given a realization $\o$ of the
medium and a realization $S$ of the dynamics, we define the sequence
of \emph{collision times} $T_n \equiv T_n(\o, S)$ via
\be \label{collisiontime}
  T_0 := 0 \,, \qquad T_n := \sum_{k=1}^n
  |\o_{S_k}-\o_{S_{k-1}}| \,, \quad \mbox{for } n \ge 1 \,.
\ee
Since the length of the $n^\mathrm{th}$ jump of the walk is
given by $|\o_{S_n}-\o_{S_{n-1}}|$, $T_n$ represents the global length
of the \tr y up to the $n^\mathrm{th}$ collision.

Finally, $X(t) \equiv X^\o(t)$ is defined by the equations
\be\label{process}
  X(t) := Y_n+ \sgn(\x_{n+1} )(t - T_n) \,, \quad \mbox{for } t \in
  [T_n , T_{n+1}) \,.
\ee
Notice that $\x_{n+1} = 0 \Leftrightarrow S_{n+1} = S_{n} \Leftrightarrow
T_{n+1} = T_n$. Therefore the definition (\ref{process}) is never used
in the case $\x_{n+1} = 0$, making the value of $\sgn(0)$ irrelevant
there. More importantly, the self-jumps of the underlying RW, i.e.,
$S_{n+1} = S_n$, are not seen by the continuous-time process
$X$. This implies that one can remove any lazy component of $S$ by
replacing the common distribution of the $\xi_j$ by the conditional
distribution of $\xi_1$ given $\xi_1 \ne 0$ (this is well defined since
$v_\xi>0$). We thus assume from now on that $Q(\xi_j=0) = 0$.

Also without loss of generality we assume that $\xi_1$, equivalently
$S$, is not supported on a proper subgroup if $\Z$. If it were
supported, say, on $d\Z$, with $d \in \N$, $d \ge 2$, one could divide
$S$ by $d$ and replace $\zeta_j$ with $\sum_{k=(j-1)d+1}^{jd}\zeta_k$.
The resulting continuous-time walk would be a simple rescaling of the
original process $X$.

The law that embraces all the random aspects of the process is the
product probability $\P := P \times Q$ on $\Oen \times \Z^\N$. The
\pr\ $\P$ is sometimes called the \emph{annealed} or
\emph{averaged} law.

Our main result is the finite-dimensional generalized CLT for the
family of processes
\be \label{x-tilde}
  \left( \bar{X}^{(q)} (s), \, s \ge 0 \right) := \left( \frac{X(sq)} {q^{1/(\a+1)}},
  \, s \ge 0 \right) \,,
\ee
where $q \in \R^+$ plays the role of the \emph{scaling parameter}. In
order to state it, we need some preliminary notation. Let $(B(t), \, t \ge 0)$
be a Brownian motion such that $B(t)$ has mean 0 and variance $v_\xi t$,
and denote by $(L_t(x), \, x \in \R)$ its corresponding local time. Let
$(Z_{\pm}(x), \, x \ge 0)$ be two i.i.d.\ c\`adl\`ag $\a$-stable processes
with independent and stationary increments such that $Z_{\pm}(0) \equiv 0$
and $Z_\pm (1)$ is distributed like $Z_1$, cf.\ beginning of the section.
Notice that the processes $Z_{\pm}$ are strictly increasing subordinators.

Assume that the processes $B$, $Z_+$ and $Z_-$ are mutually independent
and consider the process $\Delta := (\Delta(t) , \, t \ge 0 )$ given by
\be \label{Delta}
  \Delta(t) := \mu_\xi \left( \int_0^{+\infty} \!\! L_t(x) \, dZ_+(x) + \int_0^{+\infty}
  \!\! L_t(-x) \, dZ_-(x) \right) ,
\ee
with $\mu_\xi:=\mathbb E(|\xi_1|)$. The above is well-defined because
$x \mapsto L_t(x)$ is almost surely continuous and compactly supported.
Also, let us define $(Z(x), \, x \in \R)$ to be
\be
  Z(x) := \left\{
    \begin{array}{cl}
      Z_+(x), & x \ge 0 \,; \\
      -Z_-(-x), & x < 0 \,.
    \end{array}
  \right.
\ee
Observe moreover that $\Delta$ is a.s.\ continuous and strictly increasing.

\begin{theorem} \label{thm1}
  Under the law $\P$, the finite dimensional distributions of
  $( \bar{X}^{(q)} (s), \, s \ge 0 )$ converge, as $q \to \infty$, to
  those of $\left( Z \circ B \circ \Delta^{-1}(s), \, s \ge 0 \right)$.
\end{theorem}

The process $\Delta$, often called the Kesten-Spitzer process, was
introduced in \cite{ks} as the limit, for $n \to \infty$, of the processes
\be
  \left( \frac1 {n^{(1+\a) / 2\a}} \sum_{k=1}^{\lfloor nt \rfloor} \z_{S_k},
  \, t \in \R^+ \right) ,
\ee
where $\lfloor a \rfloor$ denotes the largest integer less than or equal
to $a \in \R$. The sum $\sum_k \z_{S_k}$ is referred to as a
\emph{random walk in the random scenery $\z$}; see also \cite{b1,b2}. In
the present context $\Delta$ appears as the limit of the suitably rescaled
collision time processes, as clarified by the next result, which also gives
the asymptotic behavior of $Y$, the random walk on the point process.
(The recent reference \cite{ms} also investigates limit theorems for $Y$
 in the case $\alpha \in (1,2)$, and with some restrictions. This process is of course simpler than the actual
L\'evy-Lorentz gas.)

\begin{theorem} \label{thm2}
  Under $\P$, the (joint) finite dimensional distributions of
  $$
    \left( \left( \frac{ T_{\lfloor t q \rfloor} } {q^{(1+\a) / 2\a} } \,,\,
    \frac{ Y_{\lfloor t q \rfloor} } {q^{1/2\a}} \right), \, t \ge 0 \right)
  $$
  converge, for $q \to \infty$, to the corresponding distributions of
  $( (\Delta, Z \circ B), t \ge 0)$.
\end{theorem}

As a last remark for this section, observe that neither Theorem \ref{thm1}
nor Theorem \ref{thm2} can be extended to a functional limit theorem w.r.t.\
\emph{any} Skorokhod topology, since the limit processes do not belong to
the Skorokhod space. Indeed $Z \circ B \circ \Delta^{-1}$ and $Z \circ B$
have discontinuities without one-sided limits. This can be seen as
follows: If $x_0$ is a jump discontinuity of $Z$, which exists almost surely,
and $B(t_0) = x_0$, then almost surely $B(t)$ oscillates around $x_0$,
both as $t \to t_0^+$ and as $t \to t_0^-$. Therefore neither limit
\begin{equation}
  \lim_{t \to t_0^\pm} Z \circ B (t)
\end{equation}
exists. The situation is analogous for $Z \circ B \circ \Delta^{-1}$.

Heuristically, the cause of this phenomenon is that the gaps between
targets are non-integrable i.i.d.\ variables. Therefore, from time to time,
the walker will encounter a gap of the same order as the sum of all
the gaps visited till then. The variation in position achieved upon traveling
that gap will thus be of finite order in the limit. Since, as $q \to \infty$,
time shrinks by a factor $q^{-1}$ and space by a factor
$q^{-1/(\a+1)} \ll q^{-1}$, the \tr y segments corresponding to these
exceptionally long gaps become jump discontinuities. On the other
hand, some of these gaps will be traveled on,
back and forth, a larger and larger number of times,
creating in the limit an accumulation of jump discontinuities
of the same size.

\section{Rescaled processes}
\label{sec-flow}

At the core of the proofs lies a functional CLT for the joint distribution
of all the processes involved here: the underlying \rw, the point process
and the collision time process. In this Section we state this limit theorem,
later referred to as the Main Lemma, and use it to derive Theorems
\ref{thm1} and \ref{thm2}.

Fix a value $q \in \R^+$ of the scaling parameter and let the processes
$(\bar{\o}^{(q)} (x), \, x \in \R)$, $(\bar{S}^{(q)} (t), \, t \ge 0)$,
$(\bar{T}^{(q)} (t), \, t \ge 0)$ be defined by
\begin{align}
  \label{o-bar}
  \bar{\o}^{(q)} (x) &:= \frac{ \o_{\lfloor x \sqrt{q} \rfloor} } { q^{1/2\a} } ;
  \\[4pt]
  \label{s-bar}
  \bar{S}^{(q)} (t) &:= \frac{ S_{\lfloor t q\rfloor} } { \sqrt{q} } ; \\[4pt]
  \label{t-bar}
  \bar{T}^{(q)} (t) &:= \frac{ T_{\lfloor t q\rfloor} } { q^{(\a+1)/2\a} }.
\end{align}

\bigskip
\noindent
\textbf{Convention.}\ Throughout the paper, the notation $[a,b]$ denotes
the closed  interval between the real numbers $a$ and $b$, irrespective
of their order.
\bigskip

This is our Main Lemma:

\begin{lemma} \label{main-lemma}
  Under $\P$, as $q \to \infty$, the joint process
  \begin{displaymath}
   \Big( (\bar{\o}^{(q)} (x), \, x \in \R) \,,\, (\bar{S}^{(q)} (t), \, t \ge 0)
   \,,\, (\bar{T}^{(q)} (t'), \, t' \ge 0) \Big)
  \end{displaymath}
  converges to
  \begin{displaymath}
    \Big( (Z (x), \, x \in \R) \,,\,  (B (t), \, t \ge 0)
    \,,\,  (\Delta (t'), \, t' \ge 0) \Big)
  \end{displaymath}
  w.r.t.\ the (product) $J_1$-metric on $\D(\R) \times \D(\R_0^+)
  \times \D(\R_0^+)$. Here $\D(\R)$ is the space of the \tr ies
  $\bar{\o}(x) : \R \into \R$ such that $( \bar{\o}(x) ,\, x \ge 0 )$ and
  $( -\bar{\o}(-x) ,\, x \ge 0 )$ are c\`adl\`ag; the space is endowed
  with the topology induced by the $J_1$-metric on every interval
  $[0, y]$, $y \in \R$ (see \cite[Secs.\ 12 and 16]{Bill}).
\end{lemma}

\subsection{Proof of Theorems \ref{thm1} and \ref{thm2}}

If $f: \R \into \R$ is a (not necessarily strictly) increasing \fn, let
$f^{-1}(s) := \sup \{u>0 \,:\, f(u) < s \}$ denote our choice for the
generalized inverse of $f$.

The \fn
\be
  t \mapsto \bar{T}^{( q^{2\a/(\a+1)} )} (t) = \frac{
  T_{\lfloor t q^{2\a/(\a+1)} \rfloor} } q \,,
\ee
cf.\ (\ref{t-bar}), is an increasing c\`adl\`ag \fn\ which is constant on any
interval of the form $[q^{-2\a/(\a+1)} n \,,\, q^{-2\a/(\a+1)} (n+1) )$,
$n \in \N$. Therefore, given $s \in \R^+$,
\be \label{thm1-10}
  \left( \bar{T}^{( q^{2\a/(\a+1)} )}  \right)^{-1} \!\! (s) = q^{-2\a/(\a+1)} \,
  \bar{n} ,
\ee
for some $\bar{n} \equiv \bar{n} (s,q) \in \N$. The previous two equations
imply that $T_{\bar{n}-1} < sq \le T_{\bar{n}}$, whence, in view of (\ref{Y})
and (\ref{x-tilde}),
\be \label{thm1-20}
  \bar X^{(q)}(s) \in \left[ \frac{\o_{S_{\bar{n}-1}}}
  { q^{1/(\a+1)} } \,, \frac{\o_{S_{\bar{n}}}} { q^{1/(\a+1)} } \right] .
\ee

Denoting
\be \label{def-p}
  p \equiv p(q) := q^{2\a/(\a+1)}
\ee
and using (\ref{thm1-10}) and (\ref{s-bar}), we see that
\be
  S_{\bar{n}} = \sqrt{p} \, \bar{S}^{(p)} \! \left( \left( \bar{T}^{(p)}
  \right)^{-1} (s)  \right) ,
\ee
whence, in view of (\ref{o-bar}),
\be \label{o-s-barn}
  \o_{S_{\bar{n}}} = p^{1/2\a} \, \bar{\o}^{(p)} \! \left( \bar{S}^{(p)} \left(
  \left( \bar{T}^{(p)} \right)^{-1} (s) \right) \right) .
\ee
Since $p^{1/2\a} = q^{1/(\a+1)}$, (\ref{thm1-20}) becomes
\be \label{X-interval}
  \bar{X}^{(q)}(s) \in \left[ \bar{\o}^{(p)} \circ \bar{S}^{(p)} \left( \left(
  \bar{T}^{(p)} \right)^{-1} (s)  - p^{-1} \right)  , \, \bar{\o}^{(p)} \circ
  \bar{S}^{(p)} \circ \left( \bar{T}^{(p)} \right)^{-1} (s) \right] .
\ee

Let $(q_k)_{k\in\Z^+}$ be a positive sequence of real numbers
such that $q_k \to +\infty$ as $k \to \infty$, and let $p_k := p(q_k)$ as
defined in (\ref{def-p}) above.

\bigskip
\noindent
\textbf{Assumption.}\ From now on we assume that the convergence in
the statement of Lemma \ref{main-lemma}, restricted to $q=p_k$, $k \to
\infty$, holds almost everywhere. If this is not the case, by virtue of
the Skorokhod representation theorem \cite[Thm.~3]{Dud}, there exists a
\pr\ space where it does (for processes which have the same joint
distribution as the ones used in Lemma \ref{main-lemma}): as we shall
see, the specifics of the \pr\ space are irrelevant. In particular,
up to discarding a null set of realizations, we may assume that
$B : [0,+\infty) \into \R$ is continuous and $\Delta : [0,+\infty) \into
[0,+\infty)$ is continuous, strictly increasing and bijective.
\medskip

\begin{lemma}\label{LEM0}
(Almost) surely, for all $s>0$,
$$
  \lim_{k \to \infty} \left( \bar T^{(p_k)} \right)^{-1} \! (s) = \Delta^{-1}(s) \,.
$$
\end{lemma}

\begin{proof}

Since $\Delta$ is continuous and strictly increasing, so is the inverse
$\Delta^{-1}$. For a given $s>0$ and for every sufficiently small
$\theta>0$, we set $u:=\Delta^{-1}(s)$, $x_\theta := \Delta^{-1}
(s - \theta/2)$ and $y_\theta := \Delta^{-1}(s + \theta/2)$, so that
\be
  \Delta(x_\theta)=s-\frac\theta 2< s=\Delta(u)<\Delta(y_\theta)=s +
  \frac\theta 2 \,.
\ee
Then, for any given $\e>0$, we choose $\theta>0$  such that
$u-x_\theta < \e/2$ and $y_\theta - u< \e/2$. Moreover, we set $\varsigma
:= \min\{ \theta/2, u-x_\theta, y_\theta-u, \e/2 \}$ and notice that for $k$
big enough, we have $d_{J_1, [0,u+1]} (\bar T^{(p_k)}, \Delta) < \varsigma$,
where the l.h.s.\ of the inequality denotes the Skorokhod $J_1$-distance
on the interval $[0,u+1]$. In other words, there exists a strictly increasing
homeomorphism $\varphi_k : [0,u+1] \into [0,u+1]$ such that, for all
$v\in[0,u+1]$,
\begin{align}
  |v-\varphi_k(v)| &< \varsigma \,; \\
  |\Delta(v)-\bar T^{(p_k)}(\varphi_k(v))| &< \varsigma \,.
\end{align}
Hence, for $k$ big enough, one has
\begin{align}
  \bar T^{(p_k)}(\varphi_k(y_{\theta})) &> \Delta(y_\theta)-\varsigma=s+\frac
     \theta 2-\varsigma \ge s \,;\\
  \bar T^{(p_k)}(\varphi_k(x_{\theta})) &< \Delta(x_\theta)+\varsigma=s-\frac
     \theta 2+\varsigma\le s \,,
\end{align}
whence
\begin{align}
  \left(\bar T^{(p_k)}\right)^{-1} \! (s) &\ge \varphi_k(x_{\theta})> x_{\theta}-
    \varsigma> u-\frac\e 2-\varsigma \ge u-\e \,; \\
  \left(\bar T^{(p_k)}\right)^{-1} \! (s) &\le\ \varphi_k(y_\theta)< y_{\theta}+
  \varsigma< u+\frac\e 2+\varsigma\le u+\e \,.
\end{align}
Since $\e$ was arbitrary, this proves the stated convergence.
\end{proof}
Observe now that $Z$ and $B$ are independent and that the discontinuities
of $Z$ occur at random points whose distribution has no atom in $\R$.
Therefore, for every $t\ge 0$, $Z$ is continuous at $B(t)$, equivalently, $Z$
does not have a jump at $B(t)$, with probability 1. This ensures that the
following lemma holds true for almost every realization of $(B,Z)$.

\begin{lemma}\label{LEMcomp}
Let $t>0$ and consider a realization of our processes such
that $Z$ is continuous at $B(t)$. For any sequence $(a_k) \subset
\R^+$ with $\lim_{k \to \infty} a_k = t$, we have:
$$
  \lim_{k\to\infty} \bar \o^{(p_k)}(\bar S^{(p_k)}(a_{k}))=Z(B(t)).
$$
\end{lemma}

\begin{proof}
Let $\e \in (0,1)$ and $\eta\in(0,\e)$ be such that
\be \label{lem00-1}
  \sup_{x\,:\,|x-B(t)|<2\eta} |Z(B(t))-Z(x)|<\frac \e 2\,.
\ee
Also choose $\varsigma \in (0, \eta/2)$ so that
\be
  \sup_{v \,:\, |v-t|<\varsigma}|B(t)-B(v)|<\frac \eta 2\,.
\ee
Let $k$ be big enough so that $|a_{k}-t| < \varsigma/2$ and $d_{J_1,[0,t+1]}
(\bar S^{({p_k})},B) < \varsigma/2$. Thus there exists an increasing
homeomorphisms $\varphi_k$ of $[0,\Delta^{-1}(s)+1]$ such that, for all
$v\in[0,t+1]$,
\begin{align}
  & |v-\varphi_k(v)| < \frac{\varsigma} 2 \,; \\
  & |\bar S^{({p_k})}(v) - B(\varphi_k(v))| < \frac{\varsigma} 2 \,.
\end{align}
Hence
\be
  |\bar S^{({p_k})}(a_{k})-B(t)| \le \frac{\varsigma} 2 + |B(\varphi_k(a_{k}))-B(t)|
  < \frac{\varsigma} 2+\frac\eta 2<\eta\, ,
\ee
since $|\varphi_k(a_{k})-t|\le |\varphi_k(a_{k})-a_{k}|+|a_{k}-t|<\varsigma$.

Assume moreover that $k$ is so big that
\begin{align}
  & d_{J_1,[0,|B(t)|+1]}(\bar \o^{({p_k})} ,Z) < \frac\eta 2 \,; \\
  & d_{J_1,[0,|B(t)|+1]}(\bar \o^{({p_k})} (-(\cdot)), Z({-(\cdot}))) < \frac\eta 2\,.
\end{align}
Then there exists a strictly increasing homeomorphism $\psi_k$ of
$[-|B(t)|-1,|B(t)|+1]$, with $\psi_k(0)=0$, such that, for all $w \in [-|B(t)|-1,|B(t)|+1]$,
\begin{align}
  & |w-\psi_k(w)| < \frac\eta 2 \,; \\
  & |\bar\o^{({p_k})}(w) - Z(\psi_k(w))| < \frac\eta 2 \,.
\end{align}
Therefore $|\bar\o^{({p_k})}(\bar S^{({p_k})}(a_{k})) - Z(\psi_k(\bar S^{({p_k})}(a_{k})))|
<\eta/2$ and
\be
\begin{split}
  & |\psi_k(\bar S^{({p_k})}(a_{k}))-B(t)| \\
  &\qquad \le |\psi_k(\bar S^{({p_k})}(a_{k})) - \bar S^{({p_k})}(a_{k})|+
  |\bar S^{({p_k})}(a_{k})-B(t)| \\
  &\qquad < 2\eta \,.
\end{split}
\ee
whence, by (\ref{lem00-1}),
\be
\begin{split}
  & |\bar\o^{({p_k})}(\bar S^{({p_k})}(a_{k}))-Z(B(t))| \\
  &\qquad \le \frac\eta 2 +|Z(\psi_k(\bar S^{({p_k})}(a_{k})))-Z(B(t))| \\
  &\qquad <\frac\eta 2+\frac\e 2 < \e \,.
\end{split}
\ee
We conclude that
$\lim_{k\to\infty} \, \bar\o^{({p_k})}(\bar S^{({p_k})}(a_{k}))= Z(B(t))$.
\end{proof}

We are now ready to complete the proofs of Theorems \ref{thm1}
and \ref{thm2}.

\begin{proof}[Proof of Theorem \ref{thm1}]
Let $m\in\mathbb Z^+$ and $s_1,\ldots,s_m\in[0,+\infty)$ and notice
that with probability one, $Z$ is continuous at $B(\Delta^{-1}(s_i))$,
for every $i\in\{1,\ldots,m\}$.

From the convergence stated in Lemma \ref{LEM0}, for every
$i\in\{1,\ldots,m\}$ we can apply Lemma \ref{LEMcomp} twice: first with
$a_k = (\bar T^{(p_k)})^{-1}(s_i)$ and $t = \Delta^{-1}(s_i)$, then
with $a_k = (\bar T^{(p_k)})^{-1}(s_i) - q_k^{-2\a/(\a+1)}$ and again
$t = \Delta^{-1}(s_i)$. We then obtain, for all $i=1\ldots, m$,
\be
\begin{split}
  \lim_{k\to\infty} \bar\o^{({p_k})} \!\left( \bar S^{({p_k})} \!\left( \left(
  \bar T^{(p_k)}\right)^{-1} \! (s_i) \right) \right) &= \lim_{k\to\infty}
  \bar\o^{({p_k})} \!\left( \bar S^{({p_k})} \!\left( \left(\bar T^{(p_k)}\right)^{-1}
  \! (s_i) -p_k^{-1} \right) \right) \\
  &= Z ( B ( \Delta^{-1}(s_i) ) ) \,.
\end{split}
\ee
Hence, from (\ref{X-interval}), it follows that, almost surely,
\be
  \lim_{k\to\infty} \left(\bar X^{(q_k)}(s_1), \ldots, \bar X^{(q_k)}
  (s_m) \right) = \left( Z(B(\Delta^{-1}(s_1))), \ldots, Z(B(\Delta^{-1}(s_m)))
  \right) .
\ee
The convergence then holds in distribution, for any choice of $(q_k)$.
This concludes the proof of Theorem \ref{thm1}.
\end{proof}

\medskip

\begin{proof}[Proof of Theorem \ref{thm2}] Let $m\in\mathbb \Z^+$
and $t_1,\ldots,t_m\in[0,+\infty)$. Fix an arbitrary sequence
$q_k \to +\infty$. It suffices to prove the assertion of the theorem for
$q=q_k$.

For every $i\in\{1,\ldots,m\}$ we get by definition (\ref{t-bar}) that
\be \label{thm2-1}
  \frac{T_{\lfloor t_i q \rfloor}} {q^{(\a+1)/2\a}} = \bar T^{(q)}(t_i) \,.
\ee
Also, combining definitions (\ref{o-bar})-(\ref{s-bar}) as in
(\ref{o-s-barn}) gives
\be \label{thm2-2}
  \frac{Y_{\lfloor t_i q \rfloor}} {q^{1/2\a}} = \bar\o^{(q)} \! \left(
  \bar S^{(q)} (t_i) \right) .
\ee

Now, with probability one $Z$ is continuous at $B(t_1), \ldots, B(t_m)$.
We restrict to such realizations. Fix a sequence $q_k \to +\infty$. We use
(\ref{thm2-1}) with the almost sure version of the Main Lemma
\ref{main-lemma}, for $q = q_k$, and (\ref{thm2-2}) with Lemma
\ref{LEMcomp}, for $p_k = q_k$ and $a_k = t = t_i$. This proves the a.s.\
convergence of the l.h.sides of (\ref{thm2-1})-(\ref{thm2-2}), along
the sequence $(q_k)$, to $\Delta(t_i)$ and $Z(B(t_i))$, respectively.
Considering the joint convergence for all $i\in\{1,\ldots,m\}$ and
passing to distributional convergence proves the desired assertion.
\end{proof}

\section{Proof of the Main Lemma}
\label{sec-main}

The proof of Lemma \ref{main-lemma} follows a standard argument
for the convergence of processes and will be split into two parts:
convergence of the finite-dimensional distributions and tightness.
Beyond technicalities, the main ingredient of the proof is a
representation of the collision time as a suitable random walk in
random scenery. Before getting to the proofs proper, it is worth
clarifying this idea, introducing certain quantities and giving some
estimates that will be used further on.

\subsection{Collision time and random scenery}

Let us first rewrite $T_n$ as
\be \label{colltime-RWRS}
  T_n := \sum_{y\in\Z} \mathcal N_n(y) \, \z_y,
\ee
where $\mathcal N_n(y)$ denotes the \emph{local time on the bonds}
for the underlying \rw:
\be
  \mathcal N_n(y):= \#\{k \in \{1,\ldots,n\} \,:\, [y-1,y] \subseteq
  [S_{k-1}, S_k]\} \, .
\ee
In other words, $\mathcal N_n(y)$ is the number of times, up to
the $n^\mathrm{th}$ jump, that the walk $X$ travels the
$y^\mathrm{th}$ gap $[\o_{y-1}, \o_y]$.
Notice that, given $S$, $T_n$ is a sum of independent random
variables. In particular, eq.~(\ref{colltime-RWRS}) provides an
interpretation of the collision time as a random walk in random
scenery, and suggests the convergence stated in the Main Lemma.
To formalize this connection it is useful to introduce more quantities
pertaining to the underlying \rw\ $S$, such as
the standard local time, or \emph{local time on the sites}:
 \be
  N_n(y):=\#\{k \in \{0,\ldots,n-1\} \,:\, S_k=y \} \, ;
\ee
the \emph{range}:
\be
  R_n:=\# \{y\in\Z \,:\, N_n(y)>0 \} \, ;
\ee
and the \emph{self-intersection}:
\be
  V_n:=\sum_{k,\ell=0}^{n-1} \mathbf 1_{\{S_k=S_\ell\}}=\sum_{y\in\Z}
  \sum_{k,\ell=0}^{n-1} \mathbf 1_{\{S_k=S_\ell=y\}}=\sum_{y\in\Z}
  (N_n(y))^2 \,.
\ee

We recall (see, e.g., \cite{DE}) that
\be \label{range}
  \E[R_n] = O(\sqrt{n}) \,.
\ee
Moreover, as noticed in \cite{ks}, the local limit theorem for $S$ implies
that
\be \label{expectationVn}
  \E[V_n]=\sum_{k,\ell=0}^{n-1}\mathbb P(S_k=S_\ell) = n + 2 \!\!
  \sum_{0\le k<\ell\le n-1} \!\! \mathbb P(S_{\ell-k}=0) = O(n^{3/2}) \,.
\ee
Therefore, by H\"older's inequality, for every $\beta\in (0,2]$ we have
\be \label{sumNnbeta}
\begin{split}
  \sum_{y\in\Z} \E \! \left[(N_n(y))^\beta \right] &\le \left(\E \! \left[
  \sum_{y\in\Z}{\mathbf 1}_{\{N_n(y)>0\}} \right] \right)^{1-\beta/2}
  \left(\E \! \left[ \sum_{y\in\Z} (N_n(y))^2 \right] \right)^{\beta/2} \\
  &= \left(\E [R_n] \right)^{1-\beta/2}  \left(\E [V_n] \right)^{\beta/2}
 = O \! \left( n^{(1+\beta)/2} \right) .
\end{split}
\ee

For later purposes we study certain properties of the local time $N_n$.
To state the first result, we need to introduce $\tau_n(x)$, the time of the
$n^\mathrm{th}$ visit to $x\in\Z$. We define $\tau_0(x):=0$ and, for $n \ge 1$,
\be
  \tau_n(x):=\inf\{m>\tau_{n-1}(x)\, :\, S_m=x\} \,.
\ee
Then we have the following:

\begin{lemma}\label{LEM00}
For every $\beta\in(0,1]$, there exists $K_0>0$ such that, for every
$x,y\in\mathbb Z$,
\begin{equation*}
\begin{split}
  \mathbb E \! \left[\left |N_n(x)-N_n(y)\right |^\beta\right] &\le K_0 \left(|x-y|^{\beta/2} \,
  \mathbb E \! \left[\left(N_n(x)\right)^{\beta/2} + \left(N_n(y)\right)^{\beta/2} \right]
  \right. \\
  &\quad \left. + \left(1+\left( |x-y|\log n \right)^\beta\right) \left(\mathbb P(\tau_1(x)\le n)+\mathbb P(\tau_1(y)\le n)\right)\right) .
\end{split}
\end{equation*}
\end{lemma}

\begin{proof}
As in \cite{ks}, for $x,y \in \Z$, we let $M_j(x,y)$ denote the number of visits of the
random walk $S$ to $y$ between the $j^\mathrm{th}$ and the $(j+1)^\mathrm{th}$
visit to $x$:
\be
  M_j(x,y) := \!\!\ \sum_{m=\tau_j(x)+1}^{\tau_{j+1}(x)}\mathbf 1_{\{S_m=y\}} \, .
\ee
Note that $M_j(x,y)$ is independent of the $\sigma$-algebra
$\mathcal F_{\tau_j(x)}$, where $\mathcal F_n:=\sigma(S_1, \ldots ,S_n)$.
For reasons of symmetry, it is enough to study
\be
  \mathbb E \! \left[\left |N_n(x)-N_n(y)\right|^\beta
  \mathbf 1_{\{\tau_1(x)\le\tau_1(y)\}} \right] .
\ee

Let us observe that
\be
\begin{split}
  &\left( N_n(x)-N_n(y) \right) \mathbf 1_{\{\tau_1(x)\le\tau_1(y)\}} \\
  &\qquad = \left(\mathbf 1_{\{\t_1(x)\leq n\}} + \!\! \sum_{j=1}^{N_n(x)-1}
  (1-M_j(x,y)) - \!\! \sum_{k=\tau_{N_n(x)}(x)}^{n}\mathbf 1_{\{S_k=y\}}\right)
  \mathbf 1_{\{\tau_1(x)\le\tau_1(y)\}} \, .
\end{split}
\ee
Using that, for all $a,b \in \C$ and $\beta\in [0,1]$,
\be \label{tr-in-beta}
  |a+b|^\beta \le |a|^\beta + |b|^\beta \, ,
\ee
we obtain
\begin{equation} \label{AAAA0}
  \mathbb E \! \left[(N_n(x)-N_n(y))^\beta \, \mathbf 1_{\{\tau_1(x)\le\tau_1(y)\}}
  \right] \le \P(\t_1(x)\leq n)+ \mathcal A_n(x,y)+\mathcal B_n(x,y)\, ,
\end{equation}
where
\begin{align}
  \mathcal A_n(x,y) &:= \mathbb E \! \left[\left(\sum_{j=1}^{N_n(x)-1}
  (1-M_j(x,y))\right)^{\beta}\mathbf 1_{\{\tau_1(x)\le\tau_1(y)\}}\right] ;
  \label{termA} \\
  \mathcal B_n(x,y) &:= \mathbb E \! \left[\left(\sum_{k=\tau_{N_n(x)}(x)}^{n}
  \mathbf 1_{\{S_k=y\}}\right)^\beta \mathbf 1_{\{\tau_1(x)\le\tau_1(y)\leq n\}}
  \right] .
\end{align}
%
By the independence of $(M_j(x,y), \, j\ge 1)$ and $\tau_1(x)$ we get
\begin{equation} \label{AAAA1a}
\begin{split}
  \mathcal B_n(x,y) \le \mathcal B'_n(x,y) &:= \mathbb E\left[\max_{1\le j\le n}
  |M_j(x,y)|^\beta \, \mathbf 1_{\{\tau_1(x)\le n\}}\right] \\
  &\le \mathbb P(\tau_1(x)\le n) \, \mathbb E \! \left[\max_{1\le j\le n}|M_j(x,y)|^\beta
  \right] .
\end{split}
\end{equation}
Moreover, by \cite[proof of Lem.\ 2]{ks}, it turns out that the random variables
$M_j(x,y)$ are mutually independent, with distribution given by the product of
a geometric random variable and a Bernoulli random variable, independent and
both with parameter $p(|x-y|)$, where $p: \N \into [0,1]$ is a function such that
$p(r)\sim c r^{-1}$, for some $c>0$, as $r\to +\infty$.
In particular, the variables $M_j(x,y)$ are stochastically dominated by
independent geometric variables $\widetilde M_j(x,y)$ with parameter $p(|x-y|)$,
that in turn are dominated by independent exponential variables $M'_j(x,y)$
with parameter $\lambda:=\log (1-p(|x-y|))$.
Therefore, by Jensen's inequality and since $\beta\le 1$,
\begin{equation}\label{AAAA1b}
\begin{split}
  \mathbb E \! \left[\max_{1\le j\le n}|M_j(x,y)|^\beta\right]
  &\le 1+ \mathbb E \! \left[\max_{1\le j\le n}|M'_j(x,y)|^\beta\right] \\
  &\le 1+ \mathbb E \! \left[\max_{1\le j\le n}M'_j(x,y)\right]^\beta .
\end{split}
\end{equation}
But
\be \label{AAAA1c}
\begin{split}
  \mathbb E\! \left[\max_{1\le j\le n}M'_j(x,y)\right] &= \int_0^{+\infty}
  \mathbb P \! \left(\max_{1\le j\le n}M'_j(x,y)\ge t\right) dt \\
  &= \int_0^{+\infty} \left(1-\mathbb P \! \left(\max_{1\le j\le n}M'_j(x,y)< t\right)\right)
  dt \\
  &=\int_0^{+\infty} \left(1-\left(1-e^{-\lambda t}\right)^n\right) dt \\
  &=\int_0^{+\infty} e^{-\lambda t} \sum_{k=0}^{n-1}\left(1-e^{-\lambda t}\right)^k dt \\
  &= \sum_{k=0}^{n-1}\left[\frac{\left(1-e^{-\lambda t}\right)^{k+1}}{\lambda (k+1)}
  \right]_0^{+\infty} \\
  &= \sum_{k=1}^n\frac 1{k\lambda}\le\frac{\log n + 1}\lambda \, .
\end{split}
\ee
Combining \eqref{AAAA1a}, \eqref{AAAA1b} and \eqref{AAAA1c} and using the
fact that. when $|x-y| \to +\infty$, $\lambda = \log(1-p(|x-y|)\sim p(|x-y|))\sim c|x-y|^{-1}$,
we obtain the existence of $K_1>0$ such that
\begin{equation} \label{AAAA1}
  \mathcal B_n(x,y)\le \mathcal B'_n(x,y)\le K_1 \, \mathbb P(\tau_1(x)\le n)
  \left(\log n\, |x-y|\right)^\beta\, .
\end{equation}
Concerning the term $\mathcal A_n(x,y)$ in (\ref{termA}), observe that
\begin{equation} \label{AAAA2a}
  \sum_{j=1}^{N_n(x)-1}(1-M_j(x,y))\mathbf 1_{\{\tau_1(x)\le\tau_1(y)\}} =
  \sum_{j=1}^{n}\left(1-M_j(x,y)\right)\mathbf 1_{\{\tau_j(x)\le n\}}
  \mathbf 1_{\{\tau_1(x)\le\tau_1(y)\}} \, ,
\end{equation}
and recall that, by \cite[Lem.\ 2]{ks}, the process $\left(\mathcal M_m :=
\sum_{j=1}^{m} h_j, \, m\ge 1 \right)$, with
\be
  h_j :=  \left(1-M_j(x,y)\right)\mathbf 1_{\{\tau_j(x)\le n\}}
  \mathbf 1_{\{\tau_1(x)\le\tau_1(y)\}}\, ,
\ee
is a martingale with respect to $\left (\mathcal F_{\tau_{m+1}(x)} , \, m\ge 1 \right)$,
thus implying that $\E[h_j^2|\mathcal F_{\tau_j(x)}] =
\E[M_j(x,y)^2|\mathcal F_{\tau_j(x)}]$. Therefore, using Burkholder's inequality
\cite[Thm.~2.11]{HH},
there exists $C_\beta>0$ such that
\be \label{AAAA2b}
\begin{split}
  \E \! \left[\left|\mathcal M_n\right|^\beta\right] &\le C_\beta\left(
  \E \! \left[\left(\sum_{j=1}^n\mathbb E \! \left[h_j^2|\mathcal F_{\tau_j(x)}
  \right]\right)^{\beta/2}\right] +\mathbb E \! \left[\max_{j\le n}|h_j|^\beta\right]\right) \\
  &\le C_\beta \Bigg( \E \! \left[\left(\sum_{j=1}^n \E \! \left[(M_j(x,y))^2 |
  \mathcal F_{\tau_j(x)}\right] \mathbf 1_{\{\tau_j(x)\le n\}}\right)^{\beta/2} \right] \\
  &\quad + \E \! \left[\left(1+ \max_{1\le j\le n}|M_j|^\beta\right)
  \mathbf 1_{\{\tau_1(x)\le n\}}\right] \Bigg) .
\end{split}
\ee
From  \cite[Lem.\ 2]{ks}, which ensures the existence of a constant $K_2>0$
such that $\E \! \left[(M_j(x,y))^2|\mathcal F_{\tau_j(x)}\right]\le K_2|x-y|$, we
finally get
\be
\begin{split}
  \E \! \left[\left|\mathcal M_n\right|^\beta\right] &\le K_3 \left(\E \! \left[\left(
  \sum_{j=1}^n|x-y|\mathbf 1_{\{\tau_j(x)\le n\}}\right)^{\beta/2}\right] +
  \mathcal B'_n(x,y)+ \P(\tau_1(x)\le n)\right) \\
  &\le K_3 \left(|x-y|^{\beta/2} \, \E \! \left[\left(N_n(x)\right)^{\beta/2}\right] +
  \mathcal B'_n(x,y) + \P(\tau_1(x)\le n)\right) ,
\end{split}
\ee
for some $K_3>0$. Combined with \eqref{AAAA0}, \eqref{AAAA1} and
\eqref{AAAA2a}, this ends the proof of Lemma \ref{LEM00}.
\end{proof}

\begin{lemma} \label{lem-ciao}
For any $\beta\le 1$ such that $\gamma>1+1/\beta$,
$$
     \sum_{y\in\Z} \E \! \left[ \left| \sum_{r\ge 0} \mathbb P(S_1<-r) \,
     (N_n(y)-N_n(y+r))\right|^\beta \right] = O \! \left( n^{1/2 + \beta/4} \right) ,
$$
as $n \to \infty$.
\end{lemma}

\begin{proof}
Using (\ref{tr-in-beta}) we have
\be \label{AAA2}
\begin{split}
  & \sum_{y\in\Z} \E \!\left[ \left| \sum_{r\ge 0}\mathbb P(S_1<-r) \,
  (N_n(y)-N_n(y+r))\right|^{\beta}\right] \\
  &\qquad \le \sum_{y\in\Z} \sum_{r\ge 0} (\mathbb P(S_1<-r))^{\beta} \,
  \mathbb E \! \left[\left| N_n(y)-N_n(y+r) \right|^{\beta} \right]\, .
\end{split}
\ee
By Lemma \ref{LEM00}, this quantity is dominated by
\be
\begin{split}
  & 2K_0 \sum_{y\in\Z} \sum_{r\ge 0} (\mathbb P(S_1<-r))^{\beta}
  \left(r^{\beta/2} \, \E\! \left[\left| N_n(y)\right|^{\beta/2} \right]+
  (r^\beta(\log n)^\beta+1)\P(\t_1(y)\leq n)\right) \\
  &\qquad \le 2K_0 \sum_{r\ge 0} (\mathbb P(S_1<-r))^{\beta} (r^{\beta}+1)  \left(
  \sum_{y\in\Z} \E \! \left[\left| N_n(y)\right|^{\beta/2} \right]+
  (\log n)^\beta \, \E[R_n]\right)\\
  &\qquad = O\! \left(n^{1/2 + \beta/4}\right) ,
\end{split}
\ee
where the last estimate follows from \eqref{range} and \eqref{sumNnbeta},
with $\beta$ replaced by $\beta/2$, and the following application of
the Markov inequality:
\be \label{markov-beta}
  \sum_{r\ge 0}(\mathbb P(S_1<-r))^\beta (r^\beta+1) \le ( \E[|S_1|^\gamma]
  )^\beta \sum_{r\ge 0} r^{-\gamma\beta} (r^\beta+1) \,,
\ee
because $\gamma \beta-\beta> 1$ by hypothesis.
\end{proof}

Later we will need certain inequalities for the local time on the bonds
$\cN_n$ that are analogous to the ones already seen for the local
time on the sites $N_n$. For this we will exploit the following lemma,
which estimates the $\a$-norm distance between the two types of
local times.

\begin{lemma}\label{lemma-Nn}
$$
  \sum_{y\in\Z} \E [ |\mathcal N_n(y) - \mu_\xi N_n(y)|^\a ] =
 O \! \left( n^{1/2 + \a/4} \right) .
$$
\end{lemma}

\begin{proof}
We write
\be
  \mathcal N_n(y) = \mathcal N_n^-(y)+ \mathcal N_n^+(y) \,,
\ee
where
\begin{align}
  \mathcal N_n^-(y) &:=\#\{k=0,\ldots,n-1 \,:\,S_k\ge y,\ S_{k+1}\le y- 1\} \,; \\
  \mathcal N_n^+(y) &:=\#\{k=0,\ldots,n-1 \,:\, S_{k+1}\ge y,\ S_{k}\le y- 1\}
\end{align}
denote the number of crossings of $[y-1,y]$, respectively, from right to
left and from left to right.

Observe that, for fixed $y$, the process
\be
  \left( \mathcal N^-_{n}(y) - \sum_{r\ge 0} N_n(y+r) \, \mathbb P(S_1<-r),
  \, n\in\Z^+ \right)
\ee
is a martingale with $k^\mathrm{th}$ increment given by
\be \label{incr}
  d^-_k(y) := \sum_{r\ge 0} \left( \mathbf 1_{\{S_{k}\le y-1\}} -
  \mathbb P(S_1<-r) \right) \mathbf 1_{\{S_{k-1}=y+r\}} \,.
\ee
By Burkholder's inequality \cite[Thm.~2.11]{HH}, there exists $
C_\alpha>0$ such that
\be \label{Burkholder}
  \E \! \left[ \left| \mathcal N_n^-(y) - \sum_{r\ge 0}N_n(y+r) \,
  \mathbb P(S_1<-r)\right|^\alpha \right] \le C_\alpha \!
  \left(A^-_y+ B^-_y\right) ,
\end{equation}
where
\begin{align}
  A^-_y &:= \E \! \left[ \left( \sum_{k=1}^{n} \E \! \left[ (d^-_k(y))^2
  \,\big|\, S_{k-1} \right] \right)^{\alpha/2} \right] ; \\
  B^-_y &:= \E \!\left[ \max_{k\le n}|d^-_k(y)|^\alpha \right] .
  \label{secondterm}
\end{align}

Let us estimate these terms. Using (\ref{tr-in-beta}) with $\b = \a/2<1$,
we get for the first term:
\be \label{burkho1}
\begin{split}
  \sum_{y\in\Z}A^-_y &= \sum_{y\in\Z} \E \! \left[ \left( \sum_{k=0}^{n-1}
  \sum_{r\ge 0} \P(S_1<-r) \, (1-\P(S_1<-r)) \, \mathbf 1_{\{S_k=y+r\}}
  \right)^{\a/2}\right] \\
  &\le\sum_{y\in\Z} \E \! \left[ \left(\sum_{k=0}^{n-1} \sum_{r\ge 0}
  \P(S_1<-r) \, \mathbf 1_{\{S_k= y+r\}}\right)^{\a/2} \right] \\
  &\le \sum_{y\in\Z} \E \!\left[ \left (\sum_{r\ge 0}\P(S_1<-r) \, N_n(y+r)
  \right)^{\a/2}\right] \\
  &\le \sum_{y\in\Z} \E \! \left[  \sum_{r\ge 0} \left(\P(S_1<-r) \right)^{\a/2}
  (N_n(y+r))^{\a/2} \right] \\
  &\le ( \E[|S_1|^\g] )^{\a/2}
  \sum_{r\ge 0} r^{-\a\g/2} \sum_{y\in\Z} \E \! \left[ (N_n(y))^{\a/2} \right] \\
  &=  O \!\left( n^{1/2 + \a/4} \right) ,
\end{split}
\ee
where in the last estimates we have used the Markov inequality for the
variable $|S_1|^\g$, with $\gamma>2/\a$, and estimate
(\ref{sumNnbeta}) with $\beta=\alpha/2$.

As for the second term, in view of (\ref{secondterm}) and (\ref{incr}),
we write
\be \label{burkho3}
  \sum_{y\in\Z} B^-_y \le \sum_{y\in\Z} \E \! \left[ D^-_y + E^-_y \right] ,
\ee
where
\begin{align}
  D^-_y &:= \max_{k \le n} \left| \sum_{r\ge 0}
  \mathbf 1_{\{ S_k \le y-1, \, S_{k-1} = y+r \}} \right|^\a
  = \max_{k \le n} \, \mathbf 1_{\{ S_k \le y-1<y\le S_{k-1}  \}} \,; \\
  E^-_y &:= \max_{k \le n} \left| \sum_{r\ge 0} \P( S_1 < -r) \,
  \mathbf 1_{\{ S_{k-1} = y+r \}} \right|^\a  \label{ssterm}  \\
  &= \max_{k \le n} \, \left| F_{S_1}(y-S_{k-1}) \right|^\alpha \,
  \mathbf 1_{\{ S_{k-1} \ge y\}} \, .
  \nonumber
\end{align}
Here $F_{S_1}$ is the distribution \fn\ of $S_1 = \xi_1$. The
expectation of the first of the two terms above is easily estimated
as follows, upon a relabeling of the index $k$:
\be \label{burkho2}
\begin{split}
  \sum_{y\in\Z} \E \! \left[D^-_y \right] &= \sum_{y\in\Z} \E \!\left[
  \max_{k\le n-1} \mathbf 1_{\{S_{k+1}<y\le S_k\}} \right] \\
  &\le \E \!\left[ \# \{y\in\Z \,:\, \exists k\le n-1 \mbox{ such that }
  S_{k+1}<y\le S_k \} \right] \\
  &\le \E \!\left[ \max_{k\le n-1} S_k - \min_{\ell\le n-1} S_\ell \right]
  = O \!\left(\sqrt n \right) ,
\end{split}
\ee
having used Doob's maximal inequality and the fact that $\E[S_n^2] =
v_\xi n$. As concerns the expectation of (\ref{ssterm}),
\be
\begin{split}
  \sum_{y\in\Z} \E \! \left[E^-_y \right] &\le \sum_{y\in\Z} \E \!\left[ \left(
  F_{S_1} \! \left( y - \min_{{S_{k-1} \ge y} \atop {k \le n}} S_{k-1} \right)
  \right)^\a \mathbf 1_{\{ \exists k \le n \,:\, S_{k-1} \ge y \}} \right] \\
  &\le \E \!\left[ \sum_{y = \min_{k \le n-1} S_k}^{\max_{k \le n-1} S_k}
  \!\!\!\!\ 1 \right] + \E \!\left[ \sum_{y < \min_{k \le n-1} S_k}  \!\! \left(
  F_{S_1} \! \left( y - \min_{{S_k \ge y} \atop {k \le n-1}} S_k \right)
  \right)^\a \right] \\
  &\le \E \!\left[ \max_{k\le n-1} S_k - \min_{k \le n-1} S_k + 1\right]
  + \sum_{\xi < 0} \left( F_{S_1} (\xi) \right)^\a ,
\end{split}
\ee
where again we have redefined the index $k$ in the second inequality.
Now, in the last bound above, the first term is $O(\sqrt{n})$, as in
(\ref{burkho2}), and the second term is $O(1)$, because
$(F_{S_1}(\xi))^\alpha = O( |\xi|^{-\alpha\gamma} )$ by the Markov
inequality and the fact that $\alpha\gamma>1$ by hypothesis; cf.\
(\ref{markov-beta}).

Together with (\ref{burkho3}) and (\ref{burkho2}) we conclude that
\be \label{burkho4}
  \sum_{y\in\Z} B^-_y = O \!\left( \sqrt{n} \right) ,
\ee
whence, in light of (\ref{Burkholder}) and (\ref{burkho1}),
\be
  \sum_{y\in\Z}  \E \! \left[ \left| \mathcal N_n^-(y) - \sum_{r\ge 0}N_n(y+r) \,
  \mathbb P(S_1<-r)\right|^\alpha \right] = O \!\left( n^{1/2 + \a/4}
  \right) .
\ee
Applying (\ref{tr-in-beta}) and Lemma \ref{lem-ciao}, both with $\b = \a$
(since $\gamma>2/\a>1+1/\a$), we obtain
\be
  \sum_{y\in\Z} \E \! \left[ \left| \mathcal N_n^-(y) - N_n(y) \sum_{r\ge 0}
  \mathbb P(S_1<-r)\right|^\alpha \right]
  = O \!\left( n^{1/2 + \a/4} \right) .
\ee

In the same way one can prove that
\be
  \sum_{y\in\Z}  \E \! \left[ \left| \mathcal N_n^+(y) - N_n(y) \sum_{r\ge 0}
  \mathbb P(S_1>r)\right|^\alpha \right]
  = O \!\left( n^{1/2 + \a/4} \right) .
\ee
Lemma \ref{lemma-Nn} then follows from the two previous estimates
and (\ref{tr-in-beta}) with $\b=\a$, because $\mu_\xi = \E[ |S_1| ]
= \sum_{r\ge 0} \P(|S_1|>r)$.
\end{proof}

\noindent
\textbf{Observation.}\ Lemma \ref{lemma-Nn} is the only part of the paper
where the assumption $\gamma>2/\a$ is used.

\subsection{Finite-dimensional distributions}
To prove the convergence of the finite-dimensional joint distribution
$( \bar \o^{(q)},\bar S^{(q)}, \bar T^{(q)} )$, as $q \to \infty$, we use
the method of the characteristic \fn, as in \cite{ks}.
\bigskip

\noindent
\textbf{Notation.}\
In what follows, the asymptotic equivalence $a(\theta)\sim b(\theta)$, as
$\theta \to \theta_0$, means that $\lim_{\theta \to \theta_0} a(\theta)/b(\theta)
=1$. In case $b(\theta_j)=0$, for $\theta_j \to \theta_0$, the equivalence means
that $a(\theta_j)=0$, fo all $j$. A similar notation will be used w.r.t.\ to other
limits, such as $q \to \infty$.
\bigskip

The assumption on the distribution of the gaps $\z_j$, see beginning
of Section \ref{sec-setup}, implies that the following limit exists
\begin{equation} \label{c0}
  c_0 := \lim_{z\to +\infty} z^\a \, \mathbb P(\z_1 \ge z)>0 \,.
\end{equation}
Equivalently, denoting by $\phi_{\mathcal X}$ the characteristic
function of a given random variable $\mathcal X$, the following
relation holds for $\theta\to 0$:
\begin{equation} \label{FonCar}
  1-\phi_{\z_1}(\theta) \sim -\log(\phi_{Z_1}(\theta)) =
  c_1 |\theta|^\a \left(1-i\tan\frac{\pi\a} 2 \sgn(\theta) \right) ,
\end{equation}
where $Z_1$ is the $\a$-stable variable whose basin of attraction
contains $\z_1$, cf.~Section \ref{sec-setup}, and $c_1 := \Gamma
(1-\alpha) c_0 \cos(\a\pi/2)$ \cite[\S 2.2]{IL}. Hence, there exists a
continuous
increasing function $\epsilon_0: \R_0^+ \into \R_0^+$ such that
$\epsilon_0(0)=0$ and, for all $\theta \in \R$,
\begin{equation} \label{epsilon0}
  \left| \phi_{\z_1}(\theta) - \phi_{Z_1}(\theta) \right| \le |\theta|^\a
  \epsilon_0(|\theta|) \,.
\end{equation}

\medskip

\begin{proof}[Proof of the convergence of the finite-dimensional
distributions]
For a given integer $m\ge 1$, we take $3m$ real numbers $\kappa_1,
\ldots, \kappa_m, \mu_1, \ldots, \mu_m, \theta_1, \ldots, \theta_m$,
and  $3m$ real numbers $t_1, \ldots, t_m, s_1, \ldots, s_m, v_1,
\ldots, v_m$ satisfying $0<t_1<\cdots<t_m$,  $0<v_1<\ldots<v_m$
and $s_1<\cdots<s_m$. Let us denote by $\varphi^{(q)}$ the
characteristic function of
\be
  \left( \bar\o^{(q)}(s_1), \ldots, \bar\o^{(q)}(s_m), \bar S^{(q)}(v_1), \ldots,
  \bar S^{(q)}(v_m), \bar T^{(q)}(t_1), \ldots, \bar T^{(q)}(t_m) \right)
\ee
evaluated in $(\kappa_1, \ldots, \kappa_m, \mu_1, \ldots, \mu_m, \theta_1,
\ldots, \theta_m)$. Conditioning with respect to $S$, we have
\be \label{phi1}
  \varphi^{(q)} = \E \! \left[ \E \! \left[ \left. e^{ i\sum_{j=1}^m (\kappa_j \,
  \bar\o^{(q)}(s_j) \,+\, \theta_j \, \bar T^{(q)}(t_j) ) } \right| S\, \right]
  e^{i \sum_ {j=1}^m \mu_j \, \bar S^{(q)}(v_j) } \right] .
\ee
Setting $\tilde\kappa_j:=\sgn(s_j)\kappa_j$, the inner expectation can be
rewritten as
\be \label{phi2}
\begin{split}
  &\E \! \left[ \left. e^{ i\sum_{j=1}^m (\kappa_j \, \bar\o^{(q)}(s_j) \,+\,
  \theta_j \, \bar T^{(q)}(t_j) ) } \right| S\, \right] \\
  &\qquad = \E \! \left[ \left. \exp \! \left( i\sum_{y\in\Z} \sum_{j=1}^m
  \left( \theta_j \, q^{-(1+\a)/2\a} \mathcal N_{\lfloor t_j q \rfloor}(y) +
  \tilde\kappa_j \, q^{-1/2\a} \, \mathbf 1_{I_{s_j \sqrt{q}}} (y) \right) \z_y \right)
  \right| S \right] \\
  &\qquad = \E \! \left[ \left. \prod_{y\in \Z} \phi_{\z_1} \! \left( \sum_{j=1}^m
  \left( \theta_j \, q^{-(1+\a)/2\a} \mathcal N_{\lfloor t_j q \rfloor}(y) +
  \tilde\kappa_j \, q^{-1/2\a} \, \mathbf 1_{I_{s_j \sqrt{q}}} (y) \right) \right)
  \right| S \right] ,
\end{split}
\ee
with
\be
  I_s := \left\{
    \begin{array}{cl}
      (0,s], & s \ge 0 \,; \\
      (s,0], & s < 0 \,.
    \end{array}
  \right.
\ee
Altogether we get
\be \label{phi3}
  \varphi^{(q)} \! = \E \! \left[ \prod_{y\in \Z} \phi_{\z_1} \!\! \left( \sum_{j=1}^m
  \! \left( \theta_j \, q^{-(1+\a)/2\a} \mathcal N_{\lfloor t_j q \rfloor}(y) +
  \tilde\kappa_j \, q^{-1/2\a} \, \mathbf 1_{I_{s_j \sqrt{q}}} (y) \right) \right)
  e^{i \sum_ {j=1}^m \mu_j \, \bar S^{(q)}(v_j) } \right] \!.
\ee

\bigskip
\noindent
\textbf{Convention.}\ In the remainder of this proof, for $a \in \R$, we
shall use the notation
\be
  |a|^\alpha_\sgn := |a|^\a \sgn(a) \,.
\ee
Also, we shall introduce pairs of quantities, say $A, A_\sgn$, which
are defined similarly except that a certain expression $|\cdot|^\alpha$
in the first definition is replaced by the corresponding expression
$|\cdot|^\alpha_\sgn$ in the second one. If $A, A_\sgn$ and $B, B_\sgn$
are two such pairs, we shall use the imprecise notation
\be \label{for-sigma}
  A_\s = B_\s, \quad \mbox{ for } \s \in \{ \ \ , \sgn\}
\ee
to mean that $A=B$ and $A_\sgn = B_\sgn$.
\bigskip

If we set
\be \label{def-vq}
  V^{(q)}(y) := \sum_{j=1}^m \left( \theta_j \, q^{-(1+\a)/2\a}
  \mathcal N_{\lfloor t_j q \rfloor}(y) + \tilde\kappa_j \, q^{-1/2\a} \,
  \mathbf 1_{I_{s_j \sqrt{q}}} (y) \right) ,
\ee
the following is an identity by
(\ref{FonCar}):
\be
\begin{split}
  & \prod_{y\in \Z} \phi_{Z_1} \! \left( \sum_{j=1}^m \left( \theta_j
  \, q^{-(1+\a)/2\a} \mathcal N_{\lfloor t_j q \rfloor}(y) +  \tilde\kappa_j \,
  q^{-1/2\a} \, \mathbf 1_{I_{s_j \sqrt{q}}} (y) \right) \right) \\
  &\qquad = \exp \! \left(-c_1 \sum_{y\in \Z} \left( |V^{(q)}(y)|^\a -
  i |V^{(q)}(y)|_{\sgn}^\a \, \tan\frac{\pi\a}2 \right) \right) .
\end{split}
\ee
Therefore, in view of (\ref{epsilon0}), we obtain
\be
\begin{split}
  &\left| \varphi^{(q)} - \E \! \left[ \exp \!\left(-c_1\sum_{y\in \Z} \left(
  |V^{(q)}(y)|^\a - i |V^{(q)}(y)|_{\sgn}^\a \, \tan\frac{\pi\a}2 \right) + i
  \sum_ {j=1}^m \mu_j \, \bar S^{(q)}(v_j) \right) \right] \right| \\
  &\qquad \le \E \! \left[ \sum_{y\in\Z} |V^{(q)}(y)|^\a \, \epsilon_0 \!
  \left( |V^{(q)}(y)| \right) \right] \\
  &\qquad\le \E \!\left[ \sum_{y\in\Z} |V^{(q)}(y)|^\a \right] \epsilon_0 \!
  \left( \sum_{j=1}^m \left( |\theta_j| \, t_m \, q^{(\a-1)/2\a} + |\kappa_j|
  \, q^{-1/2\a} \right) \right) ,
\end{split}
\ee
where in the last inequality we have used that $\mathcal N_{\lfloor t_j q \rfloor}
\le t_m q$. Now, from the definition of $V^{(q)}$, using (\ref{tr-in-beta})
with $\b=\a<1$, we see that the last term in the above inequality is bounded
above by
\be
   \E \! \left [\sum_{j=1}^m \left( |\theta_j|^\a \, q^{-(\a+1)/2} \sum_{y\in\Z}
   (\mathcal N_{\lfloor t_j q \rfloor}(y))^\a + |\tilde\kappa_j|^\a q^{-1/2} \,
   |s_j|^\a \, q^{\a/2} \right) \right] \epsilon_q \,,
\ee
where $\epsilon_q := \epsilon_0 ( \sum_{j=1}^m ( |\theta_j| t_j \,
q^{(\a-1)/2\a} + |\kappa_j|  q^{-1/2\a} ))$. Notice that $\epsilon_q \to 0$,
as $q \to \infty$. Now, by Lemma \ref{lemma-Nn} and
(\ref{sumNnbeta}),
\be
  \sum_{y\in\Z} \E \! \left[ (\mathcal N_n(y))^\a \right] = \sum_{y\in\Z}
  \mu_\xi^\a \,\E \!\left[ (N_n(y))^\a \right] + O \!\left( n^{1/2+\a/4} \right)
  = O \!\left( n^{1/2+\a/4} \right) .
\ee
We conclude that, for $q \to \infty$,
\be \label{marco1}
  \varphi^{(q)} \sim \E \!\left[ \exp \!\left(-c_1 \sum_{y\in \Z} \left(
  |V^{(q)}(y)|^\a - i |V^{(q)}(y)|_{\sgn}^\a \, \tan\frac{\pi\a}2 \right) + i
  \sum_{j=1}^m \mu_j \, \bar S^{(q)}(v_j) \right) \right] .
\ee

In order to take the limit of the above r.h.s.\ we define
\be
 \widetilde V^{(q)}(y):=\sum_{j=1}^m \left( \theta_j \,
  q^{-(1+\a)/2\a} \, \mu_\xi \, N_{\lfloor t_j q \rfloor}(y) + \tilde\kappa_j \,
  q^{-1/2\a} \, \mathbf 1_{I_{s_j \sqrt{q}}} (y) \right)\, .
\ee
Applying
\be \label{tr-in-sgn}
  \left| |a|_\sigma^\a - |b|_\sigma^\a \right| \le 2^{1-\a} |a-b|^\a \,,
\ee
which holds true for both $\s \in \{ \ \ , \sgn\}$, since $\alpha\le 1$,
and using \eqref{def-vq}, we obtain
\be
\begin{split}
  &\left|\sum_{y\in\mathbb Z}\left|V^{(q)}(y)\right|_\sigma^\alpha
  -\sum_{y\in\mathbb Z}\left|\widetilde V^{(q)}(y)\right|_\sigma^\alpha\right| \\
  &\qquad \le 2^{1-\a} \sum_{y\in\mathbb Z}\left|V^{(q)}(y)-\widetilde V^{(q)}(y)
  \right|^\alpha\\
  &\qquad \le 2^{1-\a}\sum_{y\in\mathbb Z}\left|\sum_{j=1}^m \theta_j \,
  q^{-(1+\a)/2\a} \left(\mathcal N_{\lfloor t_jq\rfloor}(y)-\mu_\xi \,
  N_{\lfloor t_j q \rfloor}(y)\right)\right|^\a  \\
  &\qquad \le 2^{1-\a} \sum_{y\in\mathbb Z}\sum_{j=1}^m \left|\theta_j\right|^\alpha
  q^{-(1+\a)/2} \left| \mathcal N_{\lfloor t_jq\rfloor}(y)-\mu_\xi \,
  N_{\lfloor t_j q \rfloor}(y)\right|^\a ,
\end{split}
\ee
where in the last step we have again used (\ref{tr-in-beta}) with $\b=\a$. Therefore,
by Lemma \ref{lemma-Nn},
\be \label{banana}
  \mathbb E\left[\left|\sum_{y\in\mathbb Z}\left|V^{(q)}(y)\right|_\sigma^\alpha
  -\sum_{y\in\mathbb Z}\left|\widetilde V^{(q)}(y)\right|_\sigma^\alpha\right|\right]
  = O\! \left( q^{-\a/4} \right)  .
\ee
Setting
\be
  G_\sigma^{(q)} := \sum_{y\in\Z} \left| \widetilde V^{(q)}(y) \right|_\sigma^\a
\ee
and combining (\ref{banana}) with \eqref{marco1}, we see, by means of the
Lebesgue's Dominated Convergence Theorem, that, for $q \to \infty$,
\be \label{Formulaphi}
  \varphi^{(q)} \sim \E \!\left[ \exp \! \left(-c_1 \left(G^{(q)} - i\, G_\sgn^{(q)}
  \tan \frac{\pi\a}2 \right) + i \sum_{j=1}^m \mu_j \, \bar S^{(q)}(v_j) \right)
  \right] .
\ee

We then define
\be
  G_\sigma := \int_\R \left| \sum_{j=1}^m \left( \theta_j \, \mu_\xi \,
  L_{t_j}(y) + \tilde\kappa_j \, \mathbf 1_{I_{s_j}}(y) \right)
  \right|_{\sigma}^\a dy \,.
\ee
The following lemma will be proved below.

\begin{lemma} \label{LEM1}
For fixed values of the parameters  $\{ \kappa_i, \mu_i, \theta_i, t_i,
s_i, v_i \}_{i=1}^m$, as described earlier, the random vector
$( G^{(q)}, G^{(q)}_\sgn, \sum_{j=1}^m \mu_j\, \bar S^{(q)}(v_j) )$
converges in distribution, as $q \to \infty$, to the vector
$( G, G_\sgn, \sum_{j=1}^m \mu_j \, B(v_j) )$.
\end{lemma}

It follows from (\ref{Formulaphi}), Lemma \ref{LEM1} and dominated convergence
that
\be
  \lim_{q \to \infty} \varphi^{(q)} = \E \! \left[ \exp \! \left(-c_1 \left(G - i \,
  G_\sgn \tan \frac{\pi\a}2 \right) + i \sum_{j=1}^m \mu_j \, B(v_j) \right)
  \right] ,
\ee
which is the characteristic function of
\be
  \left( Z(s_1), \ldots, Z(s_m), B(v_1), \ldots, B(v_m),  \Delta(t_1), \ldots,
  \Delta(t_m) \right)
\ee
evaluated in $(\kappa_1, \ldots, \kappa_m, \mu_1, \ldots, \mu_m, \theta_1,
\ldots, \theta_m)$. In fact,

\be
\begin{split}
  \sum_{j=1}^m\left(\kappa_jZ(s_j)+\theta_j\Delta(t_j)\right)
  &=\sum_{j=1}^m\int_0^{+\infty} \! \left(\kappa_j\mathbf 1_{I_j}(x)+\theta_jL_{t_j}(x)
  \right) dZ_+(x) \\
  &\quad +\int_0^{+\infty} \! \left(\kappa_j\mathbf 1_{I_j}(-x)+\theta_jL_{t_j}(-x)
  \right)dZ_-(x)
\end{split}
\ee
and so, conditionally to $B$,
\be
  \E\! \left[\left. e^{i\sum_{j=1}^m\left(\kappa_jZ(s_j)+\theta_j\Delta(t_j)\right)}
  \right| B\right] = \exp \left(-c_1 \left(G-iG_{sgn}\tan\frac{\pi\alpha}2\right)\right) .
\ee
\end{proof}

\begin{proof}[Proof of Lemma \ref{LEM1}]
We follow the proof of \cite[Lem.~6]{ks}. For all $q, \vartheta \in \R^+$,
$M \in \Z^+$ and $\s \in \{ \ \ , \sgn \}$, we define
\be
  \cV_\sigma(\vartheta,M,q) := \vartheta^{1-\a} \sum_{|k|\le M}
  |\cT(k,q)|_{\sigma}^\a \, ,
\ee
where
\be \label{def-tkq}
  \cT(k,q) := \sum_{j=1}^m \left( \sum_{x=\lceil k \vartheta \sqrt{q}
  \rceil+1}^{\lceil (k+1)\vartheta \sqrt{q} \rceil} \! \left( \frac1q \, \theta_j \,
  \mu_\xi \, N_{\lfloor t_j q \rfloor}(x) + \frac1{\sqrt{q}} \, \tilde\kappa_j \,
  \mathbf 1_{I_{s_j\sqrt{q}}} (x) \right) \right) ,
\ee
where $\lceil a \rceil$ is the smallest integer larger than or equal to
$a$.

As in \cite{ks}, we decompose $ G^{(q)}_\sigma - \cV_\sigma(\vartheta,M,q)$
as follows:
\be
  G_\sigma^{(q)} - \cV_\sigma(\vartheta,M,q) = U_\sigma(\vartheta,M,q)
  + W^1_\sigma(\vartheta,M,q) + W^2_\sigma (\vartheta,M,q) \,,
\ee
where
\be \label{marco2}
  U_\sigma(\vartheta,M,q) := \!\! \sum_{x\in A_{\vartheta,M,q}} \left|
  \sum_{j=1}^m \left( q^{-(1+\a)/2\a} \, \theta_j \, \mu_\xi \,
  N_{\lfloor t_j q \rfloor}(x) + q^{-1/2\a} \tilde\kappa_j
  \mathbf 1_{I_{s_j\sqrt{q}}} (x) \right) \right|_\sigma^\a ,
\ee
with $A_{\vartheta,M,q} := \Z \setminus \{\lceil -M\vartheta \sqrt q\rceil +1,
\ldots, \lceil (M+1)\vartheta \sqrt q\rceil \}$ ;
\be \label{marco3}
  W^1_\sigma(\vartheta,M,q) := \sum_{|k|\le M}\sum_{x\in E_{k,q}}
  q^{-(1+\a)/2} \, W^{1,k,q}_\sigma(x),
\ee
with $E_{k,q} := \{\lceil k \vartheta \sqrt q\rceil +1, \ldots, \lceil (k+1)
\vartheta \sqrt q \rceil \}$ and
\be
   W^{1,k,q}_\sigma(x) := \left| \sum_{j=1}^m \theta_j \, \mu_\xi \,
   N_{\lfloor t_j q\rfloor}(x) + \sqrt q \, \tilde\kappa_j
   \mathbf 1_{I_{s_j \sqrt{q}}} (x) \right|_\sigma^\a - q^\a(\#E_{k,q})^{-\a}
   \, |\cT(k,q)|_\sigma^\a \,;
\ee
and finally
\be \label{def-w2-sigma}
  W^2_\sigma (\vartheta,M,q) := \sum_{|k|\le M} \left( q^{(\a-1)/2}
  (\# E_{k,q})^{1-\a} - \vartheta^{1-\a} \right) |\cT(k,q)|_\sigma^\a \,.
\ee

\noindent
\textbf{Convention.}\
With reference to definition (\ref{marco3}), we
convene that for
those values of $\vartheta, k, q$ such that $\lceil k \vartheta \sqrt q\rceil
= \lceil (k+1)\vartheta \sqrt q \rceil$, $E_{k,q} := \emptyset$.
\bigskip

By \cite[Lem.~1, eq.~(2.11)]{ks} applied with $s=\lfloor t_mq\rfloor$ and
$A=M\vartheta/\sqrt{t_m}$, there exists a function $\eta : \R^+ \into [0,1]$,
vanishing at $+\infty$, such that
\be \label{KSL1}
\begin{split}
  &\sup_{q\ge 1/t_m} \, \mathbb P \big( U_\sigma(\vartheta,M,q) \ne 0 \big) \\
  &\qquad \le \sup_{q\ge 1/t_m} \, \mathbb P \! \left(\exists x \,:\, |x|\ge M \vartheta
  \sqrt q \,,\, N_{\lfloor t_m q \rfloor}(x) \ne 0 \right)
 + \mathbf 1_{\{\max\{|s_1|,|s_m|\} \ge M\vartheta\}} \\
  &\qquad = \eta(M\vartheta) \,.
\end{split}
\ee

Let us prove that there exist $K_1, K_2 >0 $ such that, for every
$M \in \Z^+$,
\be \label{KSL2}
  \sup_{q>1} \, \E \! \left[ |W^1_\sigma(\vartheta,M,q)| \right] \le
  K_1 M \vartheta^{1+\a/2} + K_2 \vartheta \,.
\ee
Using (\ref{tr-in-sgn}) and (\ref{tr-in-beta}) (with $\b=\a$), we have
\be
\begin{split}
  & 2^{\a-1} \, q^{-(1+\a)/2} \, \E \! \left[ |W^{1,k,q}_\sigma (x)| \right] \\
  &\qquad \le q^{-(1+\a)/2} \,\E \! \left[ \left| \sum_{j=1}^m \! \left( \theta_j
  \, \mu_\xi \, N_{\lfloor t_j q \rfloor}(x) + \sqrt{q} \, \tilde\kappa_j
  \mathbf 1_{I_{s_j\sqrt{q}}}(x) \right) - q \,\frac{\cT(k,q)} {\# E_{k,q}}
  \right|^\a \right] \\
  &\qquad \le q^{-(1+\a)/2} \,\E \! \left[ \left| \sum_{j=1}^m \theta_j
  \, \mu_\xi \left( N_{\lfloor t_j q \rfloor}(x) - \frac1 {\# E_{k,q}}
  \sum_{y \in E_{k,q}} N_{\lfloor t_j q\rfloor}(y) \right) \right|^\a \right] \\
  &\qquad\quad + q^{-(1+\a)/2} \,\E \! \left[ \left| \sum_{j=1}^m
  \sqrt{q} \, \tilde\kappa_j \left( \mathbf 1_{I_{s_j\sqrt{q}}}(x) - \frac1
  {\# E_{k,q}} \sum_{y \in E_{k,q}} \mathbf 1_{I_{s_j\sqrt{q}}}(y)
  \right) \right|^\a \right]
\end{split}
\ee
Let us sum the r.h.s.\ of the last inequality over $|k|\le M$ and
$x\in E_{k,q}$. By \cite[eq.~(3.10)]{ks}, the sum of the first term
is bounded above by some constant times $M \vartheta^{1+\a/2}$. As
for the second term, we observe that the \fn\ $E_{k,q} \ni x \mapsto
1_{I_{s_j\sqrt{q}}}(x)$ is constant for all values of $|k| \le M$ except
at most one (corresponding to certain cases where the convex hull of
$E_{k,q}$ contains $s_j \sqrt{q}$). Thus, for all non-exceptional values
of $k$ and $x \in E_{k,q}$, the term within parentheses is zero. This
and a cancelation of the exponents of $q$ show that the sum over
$|k|\le M$ and $x\in E_{k,q}$ of the last term above is of the order of
$\vartheta \sum_{j=1}^m |\kappa_j|$. So (\ref{KSL2}) is in force.

Now fix $a,b,c \in \R$ and $\eps > 0$. By (\ref{KSL2}) and the
inequality $|e^{iz} - 1| \le |z|$, for $z \in \R$, we see that one can
find values of $M, \vartheta$ such that $M \vartheta$ is arbitrarily
large but $\vartheta$ is so small that
\be \label{marco4}
  \sup_{q>1} \, \E \! \left[ \exp \!\left( i( a W^1 + b W^1_\sgn)
  (\vartheta, M, q) \right) - 1 \right] < \eps \,.
\ee
Also, let us verify that, for all $M, \vartheta$, the variable $W^2_\sigma
(\vartheta,M,q)$ converges in probability to 0 as $q \to \infty$.
For the contribution to $W^2_\sigma(\vartheta,M,q)$ coming from
the first term in $\cT(k,q)$, cf.\ (\ref{def-w2-sigma}) and (\ref{def-tkq}),
this is proved in \cite{ks} (see the argument after eq.\ (3.7) therein).
The contribution from the second term of $\cT(k,q)$ is bounded by
\be
  \sum_{|k|\le M} \left( \left| q^{(\a-1)/2} \, (\#E_{k,q})^{1-\a} -
  \vartheta^{1-\a} \right| \, \vartheta^\a \sum_{j=1}^m |\kappa_j|^\a
  \right) ,
\ee
but, for all $k$, $( q^{(\a-1)/2} \, (\#E_{k,q})^{1-\a} - \vartheta^{1-\a} ) \
\to 0$, as $q \to \infty$. Hence, for all large enough $q$ (depending
on $M, \vartheta$),
\be \label{marco5}
  \E \! \left[ \exp \!\left( i( a W^2 + b W^2_\sgn)
  (\vartheta, M, q) \right) - 1 \right] < \eps \,.
\ee
Let us introduce the random \fn\
\be \label{KSL3}
  H_\s (\vartheta,M) := \vartheta^{1-\a} \sum_{|k|\le M}
  \left| \int_{k\vartheta}^{(k+1)\vartheta} \sum_{j=1}^m \left( \theta_j \,
  \mu_\xi \, L_{t_j}(x) + \tilde \kappa_j \mathbf 1_{I_{s_j}}(x) \right) dx
  \right|_\sigma^\a .
\ee
It is easy to check that, when $M \to \infty$ and $\vartheta \to 0^+$ in
such a way that $M \vartheta \to +\infty$, $H_\s (\vartheta,M)$
converges almost surely to $G_\s$, for both labels $\s$.

Therefore, by (\ref{KSL1}), (\ref{marco4}), (\ref{marco5}) and the
limit just discussed, one can find values of $M, \vartheta$
such that, for all sufficiently large $q$,
\begin{align}
  & \E \! \left[ \left| \exp \!\left( i( a G^{(q)} + b G^{(q)}_\sgn) \right) -
  \exp \!\left( i( a \cV + b \cV_\sgn) (\vartheta, M, q) \right) \right| \right]
  < \eps \,;  \label{numero1-1} \\
  & \E \! \left[ \left| \exp \!\left( i( a H + b H_\sgn)
  (\vartheta, M) \right) - \exp \!\left( i( a G + b G_\sgn) \right) \right| \right]
  < \eps \,. \label{numero1-2}
\end{align}
Moreover, by \cite[eqs.~(2.4)-(2.8)]{ks}, when $q \to \infty$,
\be
  \left( \left(\frac 1q \sum_{x=\lceil k\vartheta \sqrt{q}\rceil}^{\lceil (k+1)\vartheta
  \sqrt{q}\rceil} N_{\lfloor t_jq\rfloor} \,,\, \bar S^{(q)}(v_j) \right) , \, j=1, \ldots, m
  \right)
\ee
converges in distribution to
\be
  \left( \left( \int_{k\vartheta}^{(k+1)\vartheta} \! L_{t_j(x)} \, dx \,,\, B(v_j) \right) ,
  \, j=1, \ldots, m \right) .
\ee
We also have the following convergence of Riemann sums:
\be
  \lim_{q \to \infty} \frac 1{\sqrt{q}}
  \sum_{x=\lceil k\vartheta \sqrt{q}\rceil}^{\lceil (k+1)\vartheta \sqrt{q}\rceil}
  \mathbf 1_{I_{s_j\sqrt{q}}}(x)=\int_{k\vartheta}^{(k+1)\vartheta}\mathbf 1_{I_{s_j}}(x)
  \, dx \, .
\ee
Therefore the vector $( \cV(\vartheta,M,q), \cV_\sgn
(\vartheta,M,q), \sum_{j=1}^m \mu_j\bar S^{(q)}(v_j) )$ converges in
distribution to $( H(\vartheta,M), H_{\sgn}(\vartheta,M),
\sum_{j=1}^m \mu_j B(v_j) )$. This means that there exists $q_0 =
q_0(M,\vartheta) > 0$ such that, for all $q > q_0$, (\ref{numero1-1})
holds together with the following:
\be \label{numero3}
\begin{split}
  &\E \! \left[ \left| \exp \! \left( i (a\cV+b\cV_\sgn)(\vartheta,M,q) + i c
  \sum_{j=1}^m \mu_j\, \bar S^{(q)}(v_j) \right) \right. \right. \\
  &\quad \left. \left. -\exp \! \left( i(aH+bH_\sgn)
  (\vartheta,M) + i c\sum_{j=1}^m \mu_j \, B(v_j) \right) \right| \right]
  <\eps \,.
\end{split}
\ee
Putting the last three inequalities together we obtain that, for all
$q > q_0$,
\be
\begin{split}
  &\left| \E \! \left[ \exp \! \left( i \left(a G^{(q)} + b G^{(q)}_\sgn + c
  \sum_{j=1}^m \mu_j \, \bar S^{(q)}(v_j) \right) \right) \right] \right. \\
  &\quad \left. - \E \! \left[ \exp \! \left( i \left(aG+b G_\sgn + c
  \sum_{j=1}^m \mu_j \, B(v_j) \right) \right)  \right] \right| <
  3\eps \,.
\end{split}
\ee
In other words, as $q \to \infty$, $( G^{(q)}, G^{(q)}_\sgn,
\sum_{j=1}^m \mu_j \, \bar S^{(q)}(v_j) )$ converges in distribution to
$(G, G_\sgn, \sum_{j=1}^m \mu_j \, B(v_j) )$.
\end{proof}

\subsection{Tightness}
It is a standard result \cite{Sko} that the sequences
$(\bar \o^{(q)}(x),\,x> 0 )_{q\in\R^+}$ and $(-\bar \o^{(q)}(-x),\,
x> 0 )_{q\in\R^+}$, cf.\ (\ref{o-bar}), converge in distribution w.r.t.\
$(\mathbb D(\R_0^+),J_1)$ to the process $Z_+$ (or $Z_-$)
defined in Section \ref{sec-setup}. In particular they are tight. The same is
true for $(\bar S^{(q)}(t),\,t\ge 0)_{q\in\R^+}$, which, due to the functional
central limit theorem, converges to $B$.

Hence it remains to prove the tightness of $(\bar T^{(q)}(t),\, t\ge
0)_{q\in\R^+}$ on $(\mathbb{D}(\mathbb R_0^+), J_1)$. Let us fix $T>0$ and
prove the tightness of $(\bar T^{(q)}(t),\, t\in[0,T])_{q\in\R^+}$ in
$(\mathbb{D}[0,T],J_1)$ following the proof of \cite[Lem.~7]{ks}.

For a given $\rho>0$,  we set $\z_{y,q} = \z_{y,q,\rho} := \z_y \,
\mathbf 1_{\{\z_y \le \rho q^{1/2\a} \}}$ and define the process
\be
  \bar T_0^{(q)}(t) = \bar T_0^{(q,\rho)}(t) := q^{-(1+\a)/2\a} \sum_{y\in\Z}
  \mathcal N_{\lfloor tq \rfloor}(y) \, \z_{y,q}\,,
\ee
that will approximate the process $  \bar T^{(q)}$.
To control the approximation error first notice that, for any $c>0$,
\be \label{approx}
\begin{split}
  &\P\! \left(\exists t\in[0,T] \,:\, \bar T^{(q)}(t) \ne \bar T_0^{(q)}(t) \right) \\
  &\qquad = \P\! \left(\exists t\in[0,T] \,:\,\sum_{y\in\Z} \mathcal N_{\lfloor tq \rfloor}(y)
  (\z_y-\z_{y,q}) \ne 0 \right) \\
  &\qquad \le \P\! \left(\exists y :  |y|\le c \sqrt{q},\, \z_y\ne \z_{y,q} \right) +
  \P\! \left(\exists t\in[0,T]\,,\exists y : |y| >c \sqrt{q},\,
  \mathcal N_{\lfloor tq \rfloor}(y)>0 \right) .
\end{split}
\ee
Now, by virtue of standard results on the maximum of a random
walk, we fix $c>0$ such that
\be
  \P\! \left(\exists t\in[0,T]\,,\exists y : |y| >c \sqrt{q},\, \mathcal N_{\lfloor tq \rfloor}(y)
  >0 \right)  = \P\! \left(\max_{k\le \lfloor Tq\rfloor}|S_k|>c\sqrt{q}\right) <
  \frac\e 4 \, .
\ee
We then choose $\rho$ such that, using (\ref{c0}) as well,
\be
  \P\! \left(\exists y :  |y|\le c \sqrt{q},\, \z_y\ne \z_{y,q} \right) \le
  3c\sqrt{q}\,\P(\z_1\ne \z_{1,q})\leq 3c c_0 \rho^{-\a} < \frac\e 4 \, .
\ee
Inserting these estimates in (\ref{approx}) we obtain
\be
  \mathbb P \! \left(\exists t\in[0,T] \,:\, \bar T^{(q)}(t) \ne \bar T_0^{(q)}(t)
  \right) < \frac \e 2 \,,
\ee
showing that it is enough to prove the tightness of $(\bar T^{(q)}_0 (t),\,
t \in[0,T])_{q\in\R^+}$ in $(\mathbb{D}[0,T],J_1)$.


We observe that, due to \cite[Thm.~13.5, eq.~(13.14)]{Bill},
it is enough to prove the existence of $K_1>0$ such that, for every $s,t$
with $0<s<t<T$,
\be \label{tightnesscondition}
  \E \! \left[ \left| \bar T^{(q)}_0(t) - \bar T^{(q)}_0(s) \right| \right]
  \le K_1 \left( q^{-1} (\lfloor tq \rfloor-\lfloor sq \rfloor) \right)^{3/4} .
\end{equation}
Indeed, together with H\"{o}lder's inequality, this would imply that for every
$0<s<r<t<T$,
\be \label{marco7}
  \E \! \left[ \left| \bar T^{(q)}_0(r) - \bar T^{(q)}_0(s) \right|^{1/2} \, \left|
  \bar T^{(q)}_0(t) - \bar T^{(q)}_0(r) \right|^{1/2} \right] \le K_1( 2( t-s ))^{3/4} \,.
\ee
The above uses that, when $t-s\ge 1/q$, $\lfloor tq\rfloor-\lfloor sq\rfloor
\le 2q(t-s)$ and, when $t-s< 1/q$, either $\lfloor rq\rfloor=\lfloor sq\rfloor$
or $\lfloor rq\rfloor=\lfloor tq\rfloor$ making the l.h.s.\ of (\ref{marco7})
null.

To verify (\ref{tightnesscondition}), we decompose $\bar T_0^{(q)}$ in
$\bar T_1^{(q)}+\bar T_2^{(q)}$, where
\begin{align}
  \bar T_1^{(q)}(t) &:= q^{-(1+\a)/2\a} \sum_{y\in\Z}
  \mathcal N_{\lfloor tq \rfloor}(y) \, \E [\z_{1,q}] \,; \\
  \bar T_2^{(q)}(t) &:=  q^{-(1+\a)/2\a} \sum_{y\in\Z}
  \mathcal N_{\lfloor tq \rfloor}(y) \, \bar\z_{y,q} \,,
\end{align}
with $\bar \z_{y,q}:=\z_{y,q} - \E [\z_{1,q}]$. Using (\ref{c0}), we see that, as
$q \to \infty$,
\be \label{expectationtruncated}
\begin{split}
  \E \! \left[ \z_{1,q} \right] &= \int_0^{+\infty} \!\! \P(\z_{1,q} > u)
  \, du = \int_0^{\rho q^{1/2\a}} \!\!\! \left( \P(\z_1 > u) - \P(\z_1 >
  \rho q^{1/2\a}) \right) du \\
  &\sim c_0 \, \frac \a {1-\a} \, \rho^{1-\a} \, q^{(1-\a)/2\a} =: K_2 \,
  q^{(1-\a)/2\a} \,,
\end{split}
\ee
for a certain positive constant $K_2$. Analogously there exists
$K_3>0$ such that
\be \label{variancetruncated}
  \E \! \left[ \z_{1,q}^2 \right] = \int_0^{+\infty} \!\!
  \P(\z_{1,q}^2 > u) \, du \le \int_0^{\rho^2 q^{1/\a}} \!\!
  \P(\z_1 > \sqrt{u}) \, du \, \sim K_3 \, q^{(2-\a)/2\a}  \, .
\ee

Now, on one hand, since $S$ and $\zeta$ are independent, we obtain
\be \label{T1}
\begin{split}
  \E\! \left[ \left|\bar T_1^{(q)}(t)-\bar T_1^{(q)}(s) \right| \right] &\le q^{-(1+\a)/2\a}
  \sum_{y\in\Z} \E\! \left[\left|\mathcal N_{\lfloor tq\rfloor}(y) -
  \mathcal N_{\lfloor sq\rfloor}(y) \right|\right] \E\!\left[\zeta_{1,q}\right] \\
  &\le q^{-(1+\a)/2\a} \!\! \sum_{j=\lfloor sq\rfloor+1}^{\lfloor tq\rfloor} \E\!\left[\left|
  S_j-S_{j-1}\right|\right] \E\!\left[\zeta_{1,q}\right] \\
  &\le 2K_2 \, q^{-1} (\lfloor tq\rfloor-\lfloor sq\rfloor) \, \E[|S_1|]\, ,
\end{split}
\ee
for every $q$ large enough, by (\ref{expectationtruncated}). On the other hand,
using in addition that the variables $\{ \bar\z_{y,q} \}_{y \in \Z}$ are mutually
orthogonal, we write
\be \label{marco6}
\begin{split}
  \E \! \left[ \left( \bar T^{(q)}_2(t) - \bar T^{(q)}_2(s) \right)^2 \right]
  &\le q^{-(1+\a)/\a} \sum_{y\in\Z} \E \! \left[ \left(
  \mathcal N_{\lfloor tq \rfloor}(y) - \mathcal N_{\lfloor sq \rfloor}(y)
  \right)^2 \right] \E \! \left[ (\bar\z_{1,q})^2 \right] \\
  &\le q^{-(1+\a)/\a} \sum_{y\in\Z} \E \! \left[ \left(
  \mathcal N_{\lfloor tq \rfloor - \lfloor sq \rfloor} (y)
  \right)^2 \right] \E \! \left[ (\z_{1,q})^2 \right] \,.
\end{split}
\ee
At this point observe that (\ref{Burkholder})-(\ref{burkho1}) and
(\ref{burkho4}) are still valid if $\alpha=2$. Hence, using
the inequality $2ab \le a^2+b^2$ and (\ref{sumNnbeta}) with
$\beta=2$, we have, for $n \to \infty$,
\be
\begin{split}
  &\sum_{y\in\Z} \E \!\left[ (\mathcal N_n(y))^2 \right] \\
  &\qquad \le 2 \sum_{y\in\Z} \E \! \left[ \left( \sum_{r\in\Z}
  \P(|S_1|>r) \, N_n(y+r) \right)^2\right] + O(n) \\
  &\qquad \le \sum_{y\in\Z} \sum_{r,s\in\Z} \P(|S_1|>r) \,
  \P(|S_1|>s) \, \E \! \left[ \left( N_n(y+r))^2 + (N_n(y+s) \right)^2
  \right] + O(n) \\
  &\qquad \le 2 \mu_\xi^2 \, \sum_{y\in\Z} \E \!\left[ (N_n(y))^2
  \right] + O(n) \\
  &\qquad = O \!\left( n^{3/2} \right) .
\end{split}
\ee
Applying the above and (\ref{variancetruncated}) in
(\ref{marco6}) yields
\be
  \E \! \left[ \left( \bar T^{(q)}_2(t) - \bar T^{(q)}_2(s) \right)^2 \right] \le
  K_4 \, q^{-3/2} \, (\lfloor tq\rfloor-\lfloor sq\rfloor)^{3/2}\, ,
\ee
for some $K_4>0$. Therefore, by Jensen's inequality,
\be
  \E \! \left[ \left| \bar T^{(q)}_2(t) - \bar T^{(q)}_2(s) \right| \right] \le
  \sqrt{K_4} \, q^{-3/4} \, (\lfloor tq\rfloor-\lfloor sq\rfloor)^{3/4}\, ,
\ee
which, combined with \eqref{T1}, gives (\ref{tightnesscondition}), concluding the
proof of the tightness of $(\bar T^{(q)})_{q\in\R^+}$ and therefore of Lemma
\ref{main-lemma}.


\footnotesize

\begin{thebibliography}{[BCLL]}

\bibitem[ACOR]{acor} \article{R.~Artuso, G.~Cristadoro, M.~Onofri,
M.~Radice} {Non-homogeneous persistent random walks and
L\'evy-Lorentz gas} {J. Stat. Mech. Theory Exp. \vol{2018}, no. 8,
083209, 13 pp}

\bibitem[B]{Bill} \book{P.~Billingsley} {Convergence of probability measure}
{Second edition. John Wiley \& Sons, Inc., New York, 1999}

\bibitem[BFK]{bfk} \article{E.~Barkai, V.~Fleurov, J.~Klafter}
{One-dimensional stochastic L\'evy-Lorentz gas} {Phys. Rev. E \vol{61}
(2000), no.~2, 1164--1169}

\bibitem[BCLL]{bcll} \article{A.~Bianchi, G.~Cristadoro, M.~Lenci,
M.~Ligab\`o} {Random walks in a one-dimensional L\'evy random
environment} {J. Stat. Phys. \vol{163} (2016), no.~1, 22--40}

\bibitem[Bo1]{b1} \article{A.~N.~Borodin} {A limit theorem for sums of
independent random variables defined on a recurrent random walk
(in Russian)} {Dokl. Akad. Nauk SSSR \vol{246} (1979), no.~4, 786--787}

\bibitem[Bo2]{b2} \article{A.~N.~Borodin} {Limit theorems for sums of
independent random variables defined on a transient random walk
(in Russian)} {Investigations in the theory of probability distributions,
IV. Zap. Nauchn. Sem. Leningrad. Otdel. Mat. Inst. Steklov.
(LOMI) \vol{85} (1979), 17--29, 237, 244}

\bibitem[BCV]{bcv} \article{R.~Burioni, L.~Caniparoli, A.~Vezzani}
{L\'evy walks and scaling in quenched disordered media}
{Phys. Rev. E \vol{81} (2010), 060101(R), 4 pp}

\bibitem[CGLS]{cgls} \article{G.~Cristadoro, T.~Gilbert, M.~Lenci,
D.~P.~Sanders} {Transport properties of L\'evy walks: an analysis in terms
of multistate processes} {Europhys. Lett. \vol{108} (2014), no.~5, 50002,
6 pp}

\bibitem[D]{Dud} \article{R.~M.~Dudley} {Distances of probability
measures and random variables} {Ann. Math. Statist. \vol{39} (1968), no.~5,
1563--1572}

\bibitem[DE]{DE} \article{A.~Dvoretzky, P.~Erd\"os} {Some problems on
random walk in space} {Proceedings of the Second Berkeley Symposium
on Mathematical Statistics and Probability, 1950. pp. 353--367. University
of California Press, Berkeley and Los Angeles, 1951}

\bibitem[HH]{HH} \book{P.~Hall, C.~C.~Heyde} {Martingale limit theory
and its application} {Probability and Mathematical Statistics. Academic
Press, Inc., New York-London, 1980}

\bibitem[IL]{IL} \book{I.~A.~Ibragimov, Yu.~V.~Linnik} {Independent and
stationary sequences of random variables} {With a supplementary chapter
by I.~A.~Ibragimov and V.~V.~Petrov. Wolters-Noordhoff Publishing,
Groningen, 1971}

\bibitem[KS]{ks} \article{H.~Kesten, F.~Spitzer} {A limit theorem related to
a new class of self-similar processes} {Z. Wahrsch. Verw. Gebiete \vol{50}
(1979), no.~1, 5--25}

\bibitem[KRS]{krs} \book{R.~Klages, G.~Radons, I.~M.~Sokolov (eds.)}
{Anomalous Transport: Foundations and Applications} {Wiley, Weinheim,
2008}

\bibitem[Le]{le} \article{P.~Levitz} {From Knudsen diffusion to Levy
walks} {Europhys. Lett. \vol{39} (1997), no.~6, 593--598}

\bibitem[MS]{ms} \article{M.~Magdziarz, W.~Szczotka} {Diffusion limit
of L\'evy-Lorentz gas is Brownian motion} {Commun. Nonlinear Sci.
Numer. Simul. \vol{60} (2018), 100--106}

\bibitem[SZU]{szu} \book{M.~Shlesinger, G.~Zaslavsky, U.~Frisch (eds.)}
{L\'evy Flights and Related Topics in Physics} {Lecture Notes in Physics
\vol{450}. Springer-Verlag, Berlin, 1995}

\bibitem[S]{Sko} \article{A.~V.~Skorohod} {Limit theorems for stochastic
processes with independent increments (in Russian)} {Teor. Veroyatnost.
i Primenen. \vol{2} (1957), 145--177}

\bibitem[VBB]{vbb} \article{A.~Vezzani, E.~Barkai, R.~Burioni}
{The Single Big Jump Principle in physical modelling} {preprint (2018),
\texttt{arXiv:1804.02932v2}}

\bibitem[ZDK] {zdk} \article{V.~Zaburdaev, S.~Denisov, J.~Klafter}
{L\'evy walks} {Rev. Mod. Phys. \vol{87} (2015), 483--530}

\end{thebibliography}
\end{document}